\documentclass[siap,onefignum,onetabnum]{siamart190516}

\usepackage{mathrsfs} 

\usepackage{lipsum}
\usepackage{amsfonts}
\usepackage{graphicx}
\usepackage{epstopdf}
\usepackage{algorithmic}
\ifpdf
  \DeclareGraphicsExtensions{.eps,.pdf,.png,.jpg}
\else
  \DeclareGraphicsExtensions{.eps}
\fi

\newsiamremark{remark}{Remark}
\newsiamremark{hypothesis}{Hypothesis}
\crefname{hypothesis}{Hypothesis}{Hypotheses}
\newsiamthm{claim}{Claim}

\headers{Nonlinear ultrasound imaging}{S. Acosta, G. Uhlmann, and J. Zhai}

\title{Nonlinear ultrasound imaging modeled by a Westervelt equation\thanks{Submitted to the editors DATE.
}}

\author{Sebastian Acosta\thanks{Department of Pediatrics-Cardiology, Baylor College of Medicine and Texas Children's Hospital, Texas, USA   (\email{sebastian.acosta@bcm.edu})}
\and Gunther Uhlmann\thanks{Department of Mathematics, University of Washington, Seattle, WA 98195, USA; Institute for Advanced Study, 
The Hong Kong University of Science and Technology, Kowloon, Hong Kong, China 
  (\email{gunther@math.washington.edu}).}
\and Jian Zhai\thanks{School of Mathematical Sciences,
  Fudan University, Shanghai, China (\email{jianzhai@fudan.edu.cn, jian.zhai@outlook.com})} }

\usepackage{amsopn}

\ifpdf
\hypersetup{
  pdftitle={Nonlinear ultrasound imaging modeled by a Westervelt equation},
  pdfauthor={S. Acosta, G. Uhlmann, and J. Zhai}
}
\fi

\begin{document}

\maketitle

\begin{abstract}
We consider the ultrasound imaging problem governed by a nonlinear wave equation of Westervelt type with variable wave speed. We show that the coefficient of nonlinearity can be recovered uniquely from knowledge of the Dirichlet-to-Neumann map. Our proof is based on a second order linearization and the use of Gaussian beam solutions to reduce the problem to the inversion of a weighted geodesic ray transform. We propose an inversion algorithm and report the results of a numerical implementation to solve the nonlinear ultrasound imaging problem in a transmission setting in the frequency domain.
\end{abstract}



\section{Introduction} \label{Section.Intro}

Nonlinear ultrasound waves play an increasingly important role in diagnostic and therapeutic medicine. Improvements are being accomplished for the ultrasonic imaging of blood perfusion in organs and tumors \cite{Thomas1998,Webb2003,Demi2014a,Szabo2004a,Cosgrove2010,Anvari2015,Harrer2003,Hedrick2005}, high-intensity focused ultrasound ablation of pathological tissues \cite{terHaar2016,Chapelon2016,Knuttel2016}, and drug/gene delivery using micro/nano agent assisted ultrasound \cite{Bettinger2016,Bouakaz2016b,Castle2016}. We are also motivated by the portability of ultrasound-based technologies which makes them ideal for monitoring patients in the operating room. For instance, there is an alarmingly high incidence of brain and kidney damage in neonates who undergo open-heart surgery to palliate congenital heart defects. Depending on the type of cardiac surgery, several studies have documented an incidence of organ injury ranging from 35\% to 75\% \cite{Andropoulos2019,Fang2019,Gaynor2015}. This problem reveals an unresolved need for better and portable imaging techniques to monitor perfusion during surgical procedures. 

The complexity of physiological media surrounding blood vessels limits the application of ultrasound to assess perfusion status. In order to reduce the influence of media heterogeneity and enhance the visualization of blood vessels, special microbubble contrast agents have been designed \cite{Webb2003,Demi2014a,Szabo2004a,Soldati2014,Cosgrove2010,Anvari2015,Demi2014,Demi2014a,Eyding2020}. Once injected in the blood stream, these enhancing agents induce a nonlinear response upon interaction with ultrasound waves. The nonlinearity generates vibration frequencies different from the isonating frequency. This effect can be measured and processed to form images of the source of nonlinearity while simultaneously filtering out some of the confounding interaction with the heterogeneous media.

The main goal of this paper is to contribute to the mathematical understanding of quantitative nonlinear ultrasound imaging. Our objective is to determine whether boundary measurements of the ultrasound field can uniquely determine the coefficient of nonlinearity in the wave equation. Our starting point is a lossless nonlinear wave equation of Westervelt type \cite{Szabo2004a,Demi2014a,Demi2014,Mischi2014} governing the propagation of waves in a bounded domain $\Omega\subset\mathbb{R}^3$ with smooth boundary $\partial \Omega$. This model has the following form,
\begin{equation}\label{eq_main}
\begin{split}
\frac{1}{c^2(x)}\partial_t^2p(t,x)-\Delta p(t,x)-\beta(x)\partial_t^2p^2(t,x)&=0,\quad\text{in }(0,T)\times\Omega,\\
p(t,x)&=f,\quad\text{on }(0,T)\times\partial\Omega,\\
p=\frac{\partial p}{\partial t}&=0,\quad\text{on }\{t=0\},
\end{split}
\end{equation}
where $p$ denotes the pressure field, $c=c(x)>0$ is the wave speed, $\beta=\beta(x)$ is the coefficient of nonlinearity, and $f$ is the isonation profile on the boundary of $\Omega$. We assume $c$ and $\beta$ are both smooth functions on $\overline{\Omega}$.

The pairing between isonating profiles and response measurements is encoded in the Dirichlet-to-Neumann (DtN) map $\Lambda$ defined as
\[
\Lambda f=\partial_\nu p\vert_{(0,T)\times\partial\Omega},
\]
where $\nu$ is the outer unit normal vector to $\partial \Omega$. Note that $\Lambda$ is well defined on $f\in \mathcal{C}^{6}((0,T)\times\partial\Omega)$ satisfying $\|f\|_{\mathcal{C}^{6}((0,T)\times\partial\Omega)}\leq \epsilon_0$ with $\epsilon_0>0$ small enough, following from the well-posedness of the initial boundary value problem \eqref{eq_main} established in Section \ref{Section.IBVP}. Hence, mathematically, the nonlinear ultrasound imaging problem can be described as the inverse problem of recovering the coefficient of nonlinearity $\beta(x)$ from the boundary measurements encoded in $\Lambda$.

Inverse problems for nonlinear hyperbolic equations have been extensively studied since the work \cite{kurylev2018inverse}, see \cite{uhlmann2019inverse,uhlmann2020inverse,hintz2020inverse,feizmohammadi2020inverse,hintz2021dirichlet,uhlmann2021inverse,lassas2021stability} and the references therein for recent progress.

The main result of this paper can be summarized in precise terms as follows.

\begin{theorem}\label{mainthm}
Consider the Riemannian metric $g=c^{-2}\mathrm{d}s^2$ on $\Omega$ associated with the wavespeed $c$. 
Assume that the Riemannian manifold $(\Omega,g)$ is non-trapping, i.e.,
\[
\mathrm{diam}_g(\Omega)=\sup\{\text{lengths of all geodesics in }(\Omega,g)\} < \infty.
\]
Also assume that $T>\mathrm{diam}_g(\Omega)$, $\partial\Omega$ is strictly convex with respect to $g$, and either of the following conditions holds
\begin{enumerate}
\item[1.] $(\Omega,g)$ has no conjugate points;
\item[2.] $(\Omega,g)$ satisfies the foliation condition introduced in \textnormal{\cite{uhlmann2016inverse}}.
\end{enumerate}
Then $\Lambda$ determines $\beta$ in $\overline{\Omega}$ uniquely.
\end{theorem}

The foliation condition on $(\Omega,g)$ means that it can be foliated by strictly convex hypersurfaces, or equivalently there exists a smooth strictly convex function $f:\Omega\rightarrow\mathbb{R}$, see \cite{paternain2019geodesic} for more details. We note that the foliation condition allows for caustics (conjugate points). For instance, in the case of radial sound speeds $c=c(r)$, if the Herglotz condition ($\frac{\partial}{\partial r}( \frac{r}{c(r)}) > 0$) is verified, then the foliation condition is satisfied. More generally, if the sound speed is written is polar coordinates $c=c(r,\hat{\textbf{x}})$ and $c$ is an increasing function of $r$, then the foliation condition is also satisfied.

In order to prove this theorem, we first establish the well-posedness of the model \eqref{eq_main} in Section \ref{Section.IBVP}. Then in Section \ref{Section.SecondLinearization} we describe our approach to solve the inverse problem based on a second order linearization of the initial boundary value problem \eqref{eq_main}. This linearization renders an integral identity (see \eqref{integral_id} below) valid for a large family of isonating profiles. In particular, Gaussian beam solutions are employed in Section \ref{Section.GaussBeams} to show the unique recovery of $\beta$ from $\Lambda$ by reducing this problem to the weighted geodesic ray transform. Section \ref{Section.Numerical} is devoted to a numerical implementation based on the analysis for a transmission imaging setting in the frequency domain. Our final remarks are given in Section \ref{Section.Conclusion}.

\section{Well-posedness for small boundary value} \label{Section.IBVP}
In this section, we establish well-posedness of the initial boundary value problem \eqref{eq_main} with small boundary value $f$. Notice that the equation is a quasilinear hyperbolic equation.
\begin{theorem}
Let $T>0$ be fixed. Assume that $f\in \mathcal{C}^{m+1}([0,T]\times\partial \Omega)$, $m\geq 5$, and it vanishes near $\{t=0\}$. Then there exists $\epsilon_0>0$ such that for $\|f\|_{\mathcal C^{m+1}}\leq \epsilon_0$, there exists a unique solution
\[
p\in \bigcap_{k=0}^m\mathcal{C}^k([0,T];\, H^{m-k}(\Omega))
\]
of equation \eqref{eq_main}. It satisfies the estimate
\[
   \sup_{t\in[0,T]}\|\partial^{m-k}_tp(t)\|_{H^{m-k}(\Omega)}\leq C\|f\|_{\mathcal{C}^{m+1}([0,T]\times\partial \Omega)}, 
\]
where $C>0$ is independent of $f$.
\end{theorem}
\begin{proof}
Fix $m\geq 5$. Assume that $f\in \mathcal{C}^{m+1}([0,T]\times\partial\Omega)$, $\|f\|_{\mathcal{C}^{m+1}([0,T]\times\partial \Omega)}\leq \epsilon_0$. First, we extend $f$ to a function $\widetilde{f}\in \mathcal{C}^{m+1}([0,T]\times\Omega)$ such that $\widetilde{f}\vert_{[0,T]\times\partial\Omega}=f$ and
\[
\|\widetilde{f}\|_{\mathcal{C}^{m+1}([0,T]\times \Omega)}\leq C\|f\|_{\mathcal{C}^{m+1}([0,T]\times\partial \Omega)}.
\]
See the proof of \cite[Theorem 2]{de2018nonlinear} for more details.
 Let $\widetilde{p}=p-\widetilde{f}$, then $\widetilde{p}$ solves the equation
\[
\frac{1}{c^2}\partial^2_t\widetilde{p}-\Delta\widetilde{p}-2\beta(\widetilde{p}+\widetilde{f})\partial^2_t\widetilde{p}=2\beta(\widetilde{p}+\widetilde{f})\partial^2_t\widetilde{f}+2\beta(\partial_t\widetilde{p}+\partial_t\widetilde{f})^2,
\]
which can be written as
\begin{alignat}{2}
\partial^2_t\widetilde{p}-\mathscr{A}(t,x,\widetilde{p},\widetilde{f})\Delta\widetilde{p}&=\mathscr{F}(t,x,\widetilde{p},\partial_t\widetilde{p},\widetilde{f},\partial_t\widetilde{f},\partial_t^2\widetilde{f}),&\quad& \text{ in }(0,T)\times \Omega, \nonumber\\
\label{eq_hom}
\widetilde{p} &= 0,&\quad& \text{ on }(0,T)\times\partial \Omega, \\
\widetilde{p}=\frac{\partial \widetilde{p}}{\partial t}&=0.&\quad&\text{ on }\{t=0\},\nonumber
\end{alignat}
where
\begin{equation*}
\begin{split}
\mathscr{A}(t,x,\widetilde{p},\widetilde{f})&=\frac{1}{c^{-2}(x)-2\beta(x)(\widetilde{p}+\widetilde{f})},\\
\mathscr{F}(t,x,\widetilde{p},\partial_t\widetilde{p},\widetilde{f},\partial_t\widetilde{f},\partial_t^2\widetilde{f})&=\frac{2\beta(\widetilde{p}+\widetilde{f})\partial^2_t\widetilde{f}+2\beta(\partial_t\widetilde{p}+\partial_t\widetilde{f})^2}{c^{-2}(x)-2\beta(x)(\widetilde{p}+\widetilde{f})}.
\end{split}
\end{equation*}
Notice that if $\widetilde{p},\widetilde{f}$ is small enough, then $c^{-2}(x)-2\beta(x)(\widetilde{p}+\widetilde{f})>0$, and $\mathscr{A},\mathscr{F}$ are well defined.

For $R>0$, define $Z(R, T)$ as the set of all functions $w$ satisfying
\[
w\in \bigcap_{k=0}^m W^{k,\infty}([0,T];\, H^{m-k}(\Omega)),\qquad
\|w\|_Z^2 := \sup_{t\in[0,T]}\sum_{k=0}^m\|\partial^k_tw(t)\|^2_{H^{m-k}}\leq R^2.
\]
Assume that $\widetilde{w}\in Z(\rho_0,T)$ for some $\rho_0>0$ small enough. Then we have
\[
  \sup_{t\in[0,T]}\sum_{k=0}^{m-1}\|\partial^{k}_t\mathscr{F}(t,x,\widetilde{w},\partial_t\widetilde{w},\widetilde{f},\partial_t\widetilde{f},\partial_t^2\widetilde{f}))\|_{H^{m-k-1}} \leq C'(\epsilon_0+\epsilon_0\rho_0+\rho_0^2),
\]
for some constant $C'>0$.
One can check that $B(t):=\mathscr{A}(x,\widetilde{w},\widetilde{f})\Delta$ satisfies Assumptions (B1)-(B3) in \cite{dafermos1985energy} if $\epsilon_0,\rho_0$ are small enough.

Given $\widetilde{w}\in Z(\rho_0,T)$, consider first the \emph{linear} initial boundary value problem
\begin{alignat}{2}
\partial^2_t\widetilde{p}-\mathscr{A}(t,x,\widetilde{w},\widetilde{f})\Delta\widetilde{p}&=\mathscr{F}(t,x,\widetilde{w},\partial_t\widetilde{w},\widetilde{f},\partial_t\widetilde{f},\partial_t^2\widetilde{f}),&\quad& \text{ in }(0,T)\times \Omega, \nonumber\\
\label{linearized_IVP}
\widetilde{p} &= 0,&\quad& \text{ on }(0,T)\times\partial \Omega, \\
\widetilde{p}=\frac{\partial \widetilde{p}}{\partial t}&=0,&\quad&\text{ on }\{t=0\}\nonumber.
\end{alignat}
By \cite[Theorem 3.1]{dafermos1985energy}, there exists a unique solution $\widetilde{p}\in \bigcap_{k=0}^m \mathcal C^k([0,T];\, H^{m-k}(\Omega))$
to \eqref{linearized_IVP}, and it satisfies the estimate
\[
  \|\widetilde{p}\|_Z \leq C(\epsilon_0 + \epsilon_0\rho_0+\rho_0^2) e^{K T},
\]
where $C,K$ are positive constants depending on the coefficients of the equation. Denote $\mathscr{T}$ to be the map which maps $\widetilde{w}\in Z(\rho_0,T)$ to the solution $\widetilde{p}$ of \eqref{linearized_IVP}.

Notice that we can take $\rho_0$ small enough and $\epsilon_0=\frac{e^{-KT}}{2C}\rho_0$ such that 
\[
C(\epsilon_0+\epsilon_0\rho_0+\rho_0^2)e^{KT}<\rho_0.
\]
Then the map $\mathscr{T}$ maps from $Z(\rho_0,T)$ to itself.

For $j=1,2$, assume that $\widetilde{w}_j\in Z(\rho_0,T)$, and $\widetilde{p}_j$ solves the equation
\begin{alignat*}{2}
\partial^2_t\widetilde{p}_j-\mathscr{A}(t,x,\widetilde{w}_j,\widetilde{f})\Delta\widetilde{p}_j&=\mathscr{F}(t,x,\widetilde{w}_j,\partial_t\widetilde{w}_j,\widetilde{f},\partial_t\widetilde{f},\partial_t^2\widetilde{f}),&\quad& \text{ in }(0,T)\times \Omega, \\
\widetilde{p}_j &= 0,&\quad& \text{ on }(0,T)\times\partial \Omega, \\
\widetilde{p}_j=\frac{\partial \widetilde{p}_j}{\partial t}&=0,&\quad&\text{ on }\{t=0\}.
\end{alignat*}
As discussed above we have $\widetilde{p}_j\in Z(\rho_0,T)$. Then $\widetilde{p}_1-\widetilde{p}_2$ solves the equation
\begin{alignat*}{2}
\left(\partial^2_t-\mathscr{A}(t,x,\widetilde{w}_1,\widetilde{f})\Delta\right)(\widetilde{p}_1-\widetilde{p}_2)&=\mathscr{H},&\quad& \text{ in }(0,T)\times \Omega, \\
\widetilde{p}_1-\widetilde{p}_2&= 0,&\quad& \text{ on }(0,T)\times\partial \Omega, \\
\widetilde{p}_1-\widetilde{p}_2=\frac{\partial (\widetilde{p}_1-\widetilde{p}_2)}{\partial t}&=0,&\quad&\text{ on }\{t=0\},
\end{alignat*}
where 
\begin{equation*}
\begin{split}
\mathscr{H}=&(\mathscr{A}(t,x,\widetilde{w}_1,\widetilde{f})-\mathscr{A}(t,x,\widetilde{w}_2,\widetilde{f}))\Delta \widetilde{p}_2\\
&+\mathscr{F}(t,x,\widetilde{w}_1,\partial_t\widetilde{w}_1,\widetilde{f},\partial_t\widetilde{f},\partial_t^2\widetilde{f})-\mathscr{F}(t,x,\widetilde{w}_2,\partial_t\widetilde{w}_2,\widetilde{f},\partial_t\widetilde{f},\partial_t^2\widetilde{f}).
\end{split}
\end{equation*}
Now equip $Z(\rho_0,T)$ with the metric $d$ defined as
\[
d(w_1,w_2)=\sup_{t\in[0,T]}\left(\|w_1(t)-w_2(t)\|_{H^1}^2+\|\partial_tw_1(t)-\partial_tw_2(t)\|^2_{L^2}\right)^{1/2}.
\]
We remark here that $(Z(\rho_0,T),d(\,\cdot\,,\,\cdot\,))$ is a complete metric space (for more details, see the footnote in \cite[Section 4]{dafermos1985energy} and also the proof of \cite[Theorem 2.2.2]{ticequasilinear}). One can show that for any $t\in[0,T]$,
\begin{equation*}
\begin{split}
\|(\mathscr{A}(t,\cdot,\widetilde{w}_1,\widetilde{f})-\mathscr{A}(t,\cdot,\widetilde{w}_2,\widetilde{f}))\Delta \widetilde{p}_2(t)\|_{L^2}\leq &C\|\widetilde{w}_1(t)-\widetilde{w}_2(t)\|_{L^2}\|\Delta\widetilde{p}_2(t)\|_{L^\infty}\\
\leq &C\|\widetilde{w}_1(t)-\widetilde{w}_2(t)\|_{L^2}\|\widetilde{p}_2(t)\|_{H^m}\\
\leq &C\rho_0\|\widetilde{w}_1(t)-\widetilde{w}_2(t)\|_{L^2},
\end{split}
\end{equation*}
and also
\begin{equation*}
\begin{split}
&\|\mathscr{F}(t,\cdot,\widetilde{w}_1,\partial_t\widetilde{w}_1,\widetilde{f},\partial_t\widetilde{f},\partial_t^2\widetilde{f})-\mathscr{F}(t,\cdot,\widetilde{w}_2,\partial_t\widetilde{w}_2,\widetilde{f},\partial_t\widetilde{f},\partial_t^2\widetilde{f})\|_{L^2}\\
\leq&C(\rho_0+\epsilon_0)(\|\widetilde{w}_1(t)-\widetilde{w}_2(t)\|_{L^2}+\|\partial_t\widetilde{w}_1(t)-\partial_t\widetilde{w}_2(t)\|_{L^2}).
\end{split}
\end{equation*}
Similar as the proof of \cite[Theorem 4.1]{dafermos1985energy}, we can show that
\[
d(\mathscr{T}(\widetilde{w}_1),\mathscr{T}(\widetilde{w}_2))^2=d(\widetilde{p}_1,\widetilde{p}_2)^2\leq C(\rho_0+\epsilon_0)Te^{KT}d(\widetilde{w}_1,\widetilde{w}_2)^2.
\]
Thus one can choose $\rho_0,\epsilon_0$ small enough such that $\mathscr{T}$ is a contraction (with respect to the metric $d$). Then by contraction mapping theorem we can conclude that the equation \eqref{eq_hom} has a unique solution $\widetilde{p}$ in $Z(\rho_0,T)$. Using
 \cite[Theorem 3.1]{dafermos1985energy} again, we have
 \[
 \widetilde{p}\in\bigcap_{k=0}^m\mathcal{C}^k([0,T];\, H^{m-k}(\Omega)).
 \]
 This completes the proof of the theorem.
\end{proof}

\section{Second order linearization} \label{Section.SecondLinearization}
We will use a second order linearization of the Dirichlet-to-Neumann map to study our inverse problem. This higher order linearization technique has been extensively used in the literature, see \cite{hintz2020inverse} and the references therein. We take $f_1,f_2,f_0\in \mathcal{C}^{m+1}([0,T]\times\partial \Omega)$ to be boundary terms such that $f_1,f_2$ vanish near $\{t=0\}$ and $f_0$ vanishes near $\{t=T\}$.

Let $p$ be the solution to \eqref{eq_main} with $f=\epsilon_1f_1+\epsilon_2f_2$, and denote
\[
\mathcal{U}^{(12)}=\frac{\partial^2}{\partial\epsilon_1\partial\epsilon_2}p\Big\vert_{\epsilon_1=\epsilon_2=0}.
\]
Then $\mathcal{U}^{(12)}$ solves
\begin{equation*}
\begin{split}
(\frac{1}{c^2(x)}\partial_t^2-\Delta)\mathcal{U}^{(12)}-2\beta(x)\partial_t^2(u_1u_2)&=0,\quad\text{in }(0,T)\times\Omega,\\
\mathcal{U}^{(12)}&=0,\quad\text{on }(0,T)\times\partial\Omega,\\
\mathcal{U}^{(12)}=\partial_t\mathcal{U}^{(12)}&=0,\quad\text{on }\{t=0\},
\end{split}
\end{equation*}
where $u_j$, $j=1,2,$ is the solution to the \textit{linear} wave equation
\begin{equation}\label{eq_vj} 
\begin{split}
\frac{1}{c^2(x)}\partial_t^2u_j(t,x)-\Delta u_j(t,x)&=0,\quad\text{in }(0,T)\times\Omega,\\
u_j(t,x)&=f_j,\quad\text{on }(0,T)\times\partial\Omega,\\
u_j=\partial_tu_j&=0,\quad\text{on }\{t=0\}.
\end{split}
\end{equation}
Additionally, let $u_0$ solve the backward wave equation
\begin{equation}\label{eq_v0} 
\begin{split}
\frac{1}{c^2(x)}\partial_t^2u_0(t,x)-\Delta u_0(t,x)&=0,\quad\text{in }(0,T)\times\Omega,\\
u_0(t,x)&=f_0,\quad\text{on }(0,T)\times\partial\Omega,\\
u_0=\partial_tu_0&=0,\quad\text{on }\{t=T\}.
\end{split}
\end{equation}
Notice that $\frac{\partial^2}{\partial\epsilon_1\partial\epsilon_2}\Lambda(\epsilon_1f_1+\epsilon_2f_2)\Big\vert_{\epsilon_1=\epsilon_2=0}=\partial_\nu\mathcal{U}^{(12)}\vert_{(0,T)\times\partial\Omega}$.
Integration by parts yields
\begin{equation}\label{integral_id}
\begin{split}
&\int_0^T\int_{\partial \Omega}\frac{\partial^2}{\partial\epsilon_1\partial\epsilon_2}\Lambda(\epsilon_1f_1+\epsilon_2f_2)\Big\vert_{\epsilon_1=\epsilon_2=0}f_0\mathrm{d}S\mathrm{d}t\\
=&\int_0^T\int_{\partial \Omega}\partial_\nu \mathcal{U}^{(12)}f_0\mathrm{d}S\mathrm{d}t\\
=&\int_0^T\int_{\Omega}\Delta\mathcal{U}^{(12)}u_0\mathrm{d}x\mathrm{d}t+\int_0^T\int_{\Omega}\nabla\mathcal{U}^{(12)}\cdot\nabla u_0\mathrm{d}x\mathrm{d}t\\
=&\int_0^T\int_{\Omega}\Delta\mathcal{U}^{(12)}u_0\mathrm{d}x\mathrm{d}t-\int_0^T\int_{\Omega}\mathcal{U}^{(12)}\Delta u_0\mathrm{d}x\mathrm{d}t\\
=&\int_0^T\int_{\Omega}\frac{1}{c^2(x)}\partial_t^2\mathcal{U}^{(12)}u_0-2\beta(x)\partial_t^2(u_1u_2)u_0\mathrm{d}x\mathrm{d}t-\int_0^T\int_{\Omega}\mathcal{U}^{(12)}\Delta u_0\mathrm{d}x\mathrm{d}t\\
=&\int_0^T\int_{\Omega}\mathcal{U}^{(12)}\left(\frac{1}{c^2(x)}\partial_t^2u_0-\Delta u_0\right)\mathrm{d}x+2\int_0^T\int_{\Omega}\beta(x)\partial_t(u_1u_2)\partial_tu_0\mathrm{d}x\mathrm{d}t\\
=&2\int_0^T\int_{\Omega}\beta(x)\partial_t(u_1u_2)\partial_tu_0\mathrm{d}x\mathrm{d}t.
\end{split}
\end{equation}
Since, for given $f_1,f_2$, the DtN map $\Lambda$ determines $\Lambda(\epsilon_1f_1+\epsilon_2f_2)$ for all small $\epsilon_1,\epsilon_2>0$, it determines $\frac{\partial^2}{\partial\epsilon_1\partial\epsilon_2}\Lambda(\epsilon_1f_1+\epsilon_2f_2)\Big\vert_{\epsilon_1=\epsilon_2=0}$. Thus for the proof of the main theorem, we only need to construct special solutions $u_j,j=0,1,2$, let $f_j=u_j\vert_{(0,T)\times\partial\Omega}$, and recover the parameters $\beta$ from the integral identity \eqref{integral_id}.

\section{Proof of the main theorem} \label{Section.GaussBeams}
We will construct Gaussian beam solutions and insert them into the integral identity \eqref{integral_id} to show the uniqueness of $\beta$ from $\Lambda$. Gaussian beams have also been used to study various inverse problems for both elliptic and hyperbolic equations \cite{katchalov1998multidimensional,bao2014sensitivity, belishev1996boundary, dos2016calderon,feizmohammadi2019timedependent,feizmohammadi2019recovery,feizmohammadi2019inverse}. The construction of Gaussian beam solutions can be carried out under the Fermi coordinates introduced below.\\

\subsection{Fermi coordinates} We first extend $(\Omega,g)$ to a slightly larger manifold $(\widetilde{\Omega},g)$, which is also non-trapping. Consider the Lorentzian manifold $(\mathbb{R}\times\widetilde{\Omega},\overline{g})$, where $\overline{g}=-\mathrm{d}t^2+g$. We introduce Fermi coordinates in a neighborhood of a null geodesic $\vartheta$ in $\mathbb{R}\times\Omega$. Assume $\vartheta(t)=(t,\gamma(t))$, where $\gamma$ is a unit-speed geodesic in the Riemannian manifold $(\Omega,g)$. Assume $\vartheta$ passes through a point $(t_0,x_0)\in (0,T)\times\Omega$, i.e. $t_0\in (0,T)$ and $\gamma(t_0)=x_0\in\Omega$. By the non-trapping condition on $(\widetilde{\Omega},g)$, we assume that $\vartheta$ joins two points $(t_-,\gamma(t_-))$ and $(t_+,\gamma(t_+))$ where  $t_-,t_+\in (0,T)$ and $\gamma(t_-),\gamma(t_+)\in\partial\Omega$. Extend $\vartheta$ to $\widetilde{M}$ such that $\gamma(t)$ is well defined on $[t_--\epsilon,t_++\epsilon]\subset (0,T)$ with $\epsilon$ a small constant. We will follow the construction of the coordinates in \cite{feizmohammadi2019timedependent}. See also \cite{kurylev2014inverse}, \cite{uhlmann2018determination}.

Choose $\alpha_2,\alpha_3$ such that $\{\dot{\gamma}(t_0),\alpha_2,\alpha_3\}$ forms an orthonormal basis for $T_{x_0}\Omega$. Let $s$ denote the arc length along $\gamma$ from $x_0$. We note here that $s$ can be positive or negative, and $(t_0+s,\gamma(t_0+s))=\vartheta(t_0+s)$. For $k=2,3$, let $e_k(s)\in T_{\gamma(t_0+s)}\Omega$ be the parallel transport of $\alpha_k$ along $\gamma$ to the point $\gamma(t_0+s)$.

Define the coordinate system $(y^0=t,y^1=s,y^2,y^3)$ through $\mathcal{F}_1:\mathbb{R}^{1+3}\rightarrow \mathbb{R}\times\widetilde{\Omega}$:
\[
\mathcal{F}_1(y^0=t,y^1=s,y^2,y^3)=(t,\exp_{\gamma(t_0+s)}\left(y^2e_2(s)+y^3e_3(s)\right)).
\]
In the new coordinates, the null-geodesic $\vartheta$ is represented by $\{t=s,y^2=y^3=0\}$. On $\vartheta$, the Lorentzian metric $\overline{g}=-\mathrm{d}t^2+g$ satisfies
\[
\overline{g}\vert_{\gamma}=-\mathrm{d}t^2+\sum_{j=1}^3(\mathrm{d}y^j)^2,\quad\text{and}\quad\frac{\partial \overline{g}_{jk}}{\partial y^i}\Big\vert_\gamma=0,~1\leq i,j,k\leq 3.
\]

Introduce the map $(z^0,z'):=(z^0,z^1,z^2,z^3)=\mathcal{F}_2(y^0=t,y^1=s,y^2,y^3)$, where
\[
z^0=\tau=\frac{1}{\sqrt{2}}(t-t_0+s),\quad z^1=r=\frac{1}{\sqrt{2}}(-t+t_0+s),\quad z^j=y^j,\, j=2,3.
\]
The Fermi coordinates $(z^0=\tau,z^1=r,z^2,z^3)$ near $\vartheta$ is given by $\mathcal{F}:\mathbb{R}^{1+3}\rightarrow \mathbb{R}\times\widetilde{\Omega}$, where $\mathcal{F}=\mathcal{F}_1\circ\mathcal{F}_2^{-1}$. Denote $\tau_\pm=\sqrt{2}(t_\pm-t_0)$.
Then on $\vartheta$ we have
\[
\overline{g}\vert_\vartheta=2\mathrm{d}\tau\mathrm{d}r+\sum_{j=2}^3(\mathrm{d}z^j)^2\quad\text{and}\quad\frac{\partial \overline{g}_{jk}}{\partial z^i}\Big\vert_\vartheta=0,~0\leq i,j,k\leq 3.
\]

\subsection{Construction of Gaussian beam solutions} We will construct asymptotic solutions of the form $u_\rho=\mathfrak{a}_\rho e^{\mathrm{i}\rho\varphi}$ on $\widetilde M$ with 
\[
\varphi=\sum_{k=0}^N\varphi_k(\tau,z'),\quad \mathfrak{a}_\rho(\tau,z')=\chi\left(\frac{|z'|}{\delta}\right)\sum_{k=0}^N\rho^{-k}a_k(\tau,z'),\quad
a_k(\tau,z')=\sum_{j=0}^Na_{k,j}(\tau,z')
\]
in a neighborhood of $\vartheta$,
\begin{equation}\label{neighbor_V}
\mathcal{V}=\bigl\{(\tau,z')\in\widetilde{M} : \tau\in \bigl[\tau_--\tfrac{\epsilon}{\sqrt{2}},\tau_++\tfrac{\epsilon}{\sqrt{2}}\bigr],\,|z'|<\delta\bigr\}.
\end{equation}
Here for each $j$, $\varphi_j$ and $a_{k,j}$ are a complex valued homogeneous polynomials of degree $j$ with respect to the variables $z^i$, $i=1,2,3$, and $\delta>0$ is a small parameter. The smooth function $\chi:\mathbb{R}\rightarrow [0,+\infty)$ satisfies $\chi(t)=1$ for $|t|\leq\frac{1}{4}$ and $\chi(t)=0$ for $|t|\geq \frac{1}{2}$. 

Notice that
\[
(-\frac{\partial^2}{\partial t^2}+c^2\Delta) u=(-\frac{\partial^2}{\partial t^2}+\Delta_g)u-c^2\sum_{i=1}^3\frac{\partial\log c}{\partial x^i}\frac{\partial u}{\partial x^i}=\square_{\overline{g}}u-\langle\mathrm{d}(\log c),\mathrm{d} u\rangle_{\overline{g}}.
\]
We have
\begin{equation}
(\frac{\partial^2}{\partial t^2}-c^2\Delta)(\mathfrak{a}_\rho e^{\mathrm{i}\rho\varphi})=e^{\mathrm{i}\rho\varphi}(\rho^2(\mathcal{S}\varphi)\mathfrak{a}_\rho-\mathrm{i}\rho\mathcal{T}\mathfrak{a}_\rho+\square_{\overline{g}}\mathfrak{a}_\rho-\langle \mathrm{d}(\log c),\mathrm{d}\mathfrak{a}_\rho\rangle_{\overline{g}}), 
\end{equation}
where
\begin{equation*}
\begin{split}
  &\quad \mathcal{S}\varphi=\langle \mathrm{d}\varphi,\mathrm{d}\varphi\rangle_{\overline{g}}, \\
  &\quad \mathcal{T}a=2\langle \mathrm{d}\varphi,\mathrm{d}a\rangle_{\overline{g}}-\langle \mathrm{d}(\log c),\mathrm{d}\varphi\rangle_{\overline{g}}a-(\square_{\overline{g}}\varphi)a.
\end{split}
\end{equation*}
We need to construct $\varphi$ and $\mathfrak{a}_\rho$ such that
\begin{equation}\label{S_cond}
\begin{split}
&\frac{\partial^\Theta}{\partial z^\Theta}(\mathcal{S}\varphi)(\tau,0)=0,\\
&\frac{\partial^\Theta}{\partial z^\Theta}(\mathcal{T}a_0)(\tau,0)=0, \\
&\frac{\partial^\Theta}{\partial z^\Theta}(-\mathrm{i}\mathcal{T}a_k+\square_{\overline{g}} a_{k-1}-\langle \mathrm{d}(\log c),\mathrm{d}a_{k-1}\rangle_{\overline{g}})(\tau,0)=0,
\end{split}
\end{equation}
for $\Theta=(\Theta_0=0,\Theta_1,\Theta_2,\Theta_3)$ with $|\Theta|\leq N$. For more details we refer to \cite{feizmohammadi2019recovery}. Following \cite{feizmohammadi2019timedependent}, we take
\[
\varphi_0=0,\quad \varphi_1=z^1,\quad
\varphi_2(\tau,z')=\sum_{1\leq i,j\leq 3}H_{ij}(\tau)z^iz^j.
\]
Here $H$ is a symmetric matrix with $\Im H(\tau)>0$; the matrix $H$ satisfies a Riccati ODE,
\begin{equation}\label{Ricatti}
\frac{\mathrm{d}}{\mathrm{d}\tau}H+HCH+D=0,\quad\tau\in \bigl(\tau_--\tfrac{\epsilon}{2},\tau_++\tfrac{\epsilon}{2}\bigr),\quad H(0)=H_0,\text{ with }\Im H_0>0,
\end{equation}
where $C$, $D$ are matrices with $C_{11}=0$, $C_{ii}=2$, $i=2,3$, $C_{ij}=0$, $i\neq j$ and $D_{ij}=\frac{1}{4}(\partial_{ij}^2\overline{g}^{11})$.
 
 \begin{lemma}[\text{\cite[Lemma 3.2]{feizmohammadi2019timedependent}}]
 The Ricatti equation \eqref{Ricatti} has a unique solution. Moreover the solution $H$ is symmetric and $\Im (H(\tau))>0$ for all $\tau\in (\tau_--\frac{\delta}{2},\tau_++\frac{\delta}{2})$. For solving the above Ricatti equation, one has $H(\tau)=Z(\tau)Y(\tau)^{-1}$, where $Y(\tau)$ and $Z(\tau)$ solve the ODEs
\begin{equation*}
 \begin{split}
& \frac{\mathrm{d}}{\mathrm{d}\tau}Y(\tau)=CZ(\tau),\quad Y(\tau_-)=Y_0,\\
&\frac{\mathrm{d}}{\mathrm{d}\tau}Z(\tau)=-D(\tau)Y(\tau),\quad Z(\tau_-)=Y_1=H_0Y_0.
 \end{split}
\end{equation*}
 In addition, $Y(\tau)$ is nondegenerate.
 \end{lemma}

 \begin{lemma}[\text{\cite[Lemma 3.3]{feizmohammadi2019timedependent}}]\label{lemma_H0}
 The following identity holds:
\[
 \det(\Im(H(\tau))|\det(Y(\tau))|^2=C_0
\]
 with $C_0$ independent of $\tau$.
 \end{lemma}
 
  We see that the matrix $Y(\tau)$ satisfies
  \begin{equation}\label{eq_Y}
  \frac{\mathrm{d}^2}{\mathrm{d}\tau^2}Y+CD Y=0,\quad Y(0)=Y_0,\quad \frac{\mathrm{d}}{\mathrm{d}\tau}Y(0)=CY_1.
  \end{equation}
To construct the amplitude, first notice that $\mathcal{T}_0a_0=0$ simplifies to
\[
(2\partial_\tau-\partial_\tau(\log c))a_0-(\square_{\overline{g}}\varphi)a_0=0.
\]
Using the fact
\[
\square_{\overline{g}}\varphi=\frac{\mathrm{d}}{\mathrm{d}\tau}\log(\det(Y(\tau))),
\]
we can set
\[
a_0\vert_{\vartheta(\tau)}=\det(Y(\tau))^{-1/2}c(\tau,0)^{-1/2}.
\]

Fix $k\geq 4$. As in \cite{feizmohammadi2019recovery}, we have the following estimate by the construction of $u_\rho$ (cf.\ \eqref{S_cond})
\begin{equation}\label{est_remainder}
\|(c^{-2}\partial_t^2-\Delta) u_\rho\|_{H^k((0,T)\times\Omega)}\leq C\rho^{-K},\qquad K=\frac{N+1-k}{2}-1.
\end{equation}
We construct Gaussian beam solutions of the forms
\begin{equation}\label{uj0}
u^{(1)}_\rho=u^{(2)}_\rho=\mathfrak{a}_\rho e^{\mathrm{i}\rho\varphi},\quad u^{(0)}_{2\rho}=\overline{\mathfrak{a}_{2\rho}}e^{-2\mathrm{i}\rho\overline{\varphi}},
\end{equation}
concentrating near the same null-geodesic $\vartheta(t)=(t,\gamma(t))$ as in \cite{feizmohammadi2019timedependent}. The parameter $\delta$ can be taken small enough such that $u^{(j)}_\rho=0$, $j=1,2$ near $\{t=0\}$ and $u^{(0)}_{2\rho}=0$ near $\{t=T\}$.

For $j=1,2$, we let
\[
u_j=u^{(j)}_\rho+R^{(j)}_\rho
\]
where $R^{(j)}_\rho$ be the solution to the initial boundary value problem
\begin{alignat*}{2}
(c^{-2}\partial_t^2-\Delta) R_\rho^{(j)}&=-(c^{-2}\partial_t^2-\Delta) u_\rho^{(1)}&\quad&\text{ in }(0,T)\times \Omega,\\
R_\rho^{(j)}&=0&\quad&\text{ on }  (0,T)\times\partial \Omega,\\
R_\rho^{(j)}=\partial_tR_\rho^{(j)}&=0&\quad&\text{ on }\{t=0\}.
\end{alignat*}
Invoking \eqref{est_remainder}, the remainder $R_\rho^{(j)}$ satisfies the estimate
\[
\|R_\rho^{(j)}\|_{H^{k+1}((0,T)\times\Omega)}\leq C\rho^{-K}.
\]
Using Sobolev embedding, we can choose $N$ large enough such that
\[
\|R_\rho^{(j)}\|_{\mathcal{C}((0,T)\times\Omega)}+\|\partial_tR_\rho^{(j)}\|_{\mathcal{C}((0,T)\times\Omega)}\leq C\rho^{-3/2}.
\]
We note that $u_j=u_\rho^{(j)}+R_\rho^{(j)}$ is the solution to \eqref{eq_vj} with boundary value $f_j=u_\rho^{(j)}\vert_{(0,T)\times\partial \Omega}$.

Similarly, we can construct a solution to \eqref{eq_v0} of the form $u_0=u^{(0)}_{2\rho}+R^{(0)}_{2\rho}$, where $u^{(0)}_{2\rho}$ is constructed in \eqref{uj0}. For this purpose, we only need to take the remainder term $R_{2\rho}^{(0)}$ to be the solution to the backward wave equation
\begin{alignat*}{2}
(c^{-2}\partial_t^2-\Delta) R_{2\rho}^{(0)}&=-(c^{-2}\partial_t^2-\Delta) u_{2\rho}^{(0)}&\quad&\text{ in }(0,T)\times \Omega,\\
R_{2\rho}^{(0)}&=0&\quad&\text{ on }  (0,T)\times\partial \Omega,\\
R_{2\rho}^{(0)}=\partial_tR_{2\rho}^{(0)}&=0&\quad&\text{ on }\{t=T\}.
\end{alignat*}
Now $u_0$ is the solution to \eqref{eq_v0} with $f_0= u_{2\rho}^{(0)}\vert_{(0,T)\times\partial \Omega}$.

\subsection{Proof of Theorem \ref{mainthm}}
Now extend $\beta$ from $\Omega$ to $\widetilde{\Omega}$ such that $\beta=0$ in $\widetilde{\Omega}\setminus\Omega$.
Upon using the Gaussian beam solutions constructed in previous section in the integral identity \eqref{integral_id}, we see that the DtN map $\Lambda$ determines
\begin{equation*}
\begin{split}
&\rho^{-1/2}\int_0^T\int_{\Omega}\beta(x)\partial_t(u_1u_2)\partial_tu_0\mathrm{d}x\mathrm{d}t\\
=&\rho^{-1/2}\int_0^T\int_{\Omega}\beta(x)e^{-4\rho\Im{\varphi}}\chi^3(2\mathrm{i}\rho)\partial_t\varphi a_0a_0(-2\mathrm{i}\rho\overline{\partial_t\varphi})\overline{a_0}\mathrm{d}x\mathrm{d}t+\mathcal{O}(\varrho^{-1})\\
=&4\rho^{3/2}\int_0^T\int_{\Omega}\beta(x)e^{-4\rho\Im\varphi}\chi^3|\partial_t\varphi|^2a_0a_0\overline{a_0}\mathrm{d}x\mathrm{d}t+\mathcal{O}(\varrho^{-1})\\
=&4\rho^{3/2}\int_{\tau_--\frac{\epsilon}{\sqrt{2}}}^{\tau_++\frac{\epsilon}{\sqrt{2}}}\int_{|z'|<\delta}\beta e^{-4\varrho\Im\varphi}\chi^3(\frac{|z'|}{\delta})|\partial_t\varphi|^2a_0a_0\overline{a_0}c^{3}\mathrm{d}z'\wedge\mathrm{d}\tau+\mathcal{O}(\varrho^{-1}).
\end{split}
\end{equation*}
Here we used the fact that the Euclidean volume form is $\mathrm{d}x\mathrm{d}t=c^3\mathrm{d}\tau\wedge\mathrm{d}x'$ under the Fermi coordinates. 
Notice that
\[
|\partial_t\varphi|^2a_0a_0\overline{a_0}\big\vert_{\vartheta}=C_0|\det Y(\tau)|^{-1}(\det Y(\tau))^{-1/2}c^{-3/2},
\]
and thus we have, using the method of stationary phase,
\begin{equation*}
\begin{split}
&\rho^{3/2}\int_{|z'|<\delta}\beta e^{-4\varrho\Im\varphi}\chi^3(\frac{|z'|}{\delta})|\partial_t\varphi|^2a_0a_0\overline{a_0}c^{3}\mathrm{d}z'\\
=&C_0(\det Y(\tau))^{-1/2}\beta(\tau,0)c^{3/2}+\mathcal{O}(\rho^{-1}).
\end{split}
\end{equation*}

Letting $\rho\rightarrow+\infty$, we can extract the line integral
\[
\int_{\vartheta(\tau)}\beta(\vartheta(\tau))c^{3/2}(\vartheta(\tau))(\det Y(\tau))^{-1/2}\mathrm{d}\tau
\]
from the DtN map $\Lambda$ by \eqref{integral_id}. To show the unique determination of $\beta$ from $\Lambda$, we only need to show that the above line integral uniquely determines $\beta$.

Since the Lorentzian metric $\overline{g}=-\mathrm{d}t^2+g$ is of the product form, we have
\[
\overline{g}_{00}=\overline{g}\left(\frac{\partial}{\partial z^0},\frac{\partial}{\partial z^0}\right)=\overline{g}\left(\frac{1}{\sqrt{2}}\frac{\partial}{\partial y^1}+\frac{1}{\sqrt{2}}\frac{\partial}{\partial y^0},\frac{1}{\sqrt{2}}\frac{\partial}{\partial y^1}+\frac{1}{\sqrt{2}}\frac{\partial}{\partial y^0}\right)=\frac{1}{2}g_{11}-\frac{1}{2}.
\]
Here $g_{11}=g(\frac{\partial}{\partial y^1},\frac{\partial}{\partial y^1})$. Using \cite[Lemma 3.4 and Corollary 3.5]{feizmohammadi2019timedependent}, we have
\[
\frac{\partial^2\overline{g}^{11}}{\partial z^1\partial z^i}\Bigg\vert_{\vartheta(t)}=0
\]
for any $i=1,2,3$, and
\[
\frac{\partial^2\overline{g}^{11}}{\partial z^i\partial z^j}\Big\vert_{\vartheta(t)}=-\frac{\partial^2\overline{g}_{00}}{\partial z^i\partial z^j}\Big\vert_{\vartheta(t)}=-\frac{1}{2}\frac{\partial^2g_{11}}{\partial y^i\partial y^j}\Big\vert_{\gamma(t)}=-R_{1i1j}\vert_{\gamma(t)},
\]
for $i,j=1,2,3$, where $R$ is the Riemann curvature tensor on the Riemannian manifold $(\Omega,g)$.

We can take $Y_{11}\equiv 1$, $Y_{1j}=Y_{i1}=0$ for all $i,j=1,2,3$.
Notice $\tau=\sqrt{2}(t-t_0)$ and $z'=0$ on $\vartheta$, thus $Y$ solves the equation
\[
\frac{\mathrm{d}^2}{\mathrm{d}t^2}Y=-4DY,
\]
which is equivalent to
\[
\frac{\mathrm{d}^2}{\mathrm{d}t^2}Y_{ij}-g^{k\ell}R_{1i1k}Y_{\ell j}=0.
\]
%

Denote $\widetilde{Y}=(Y_{\alpha}^{\phantom{\alpha}\beta})_{\alpha,\beta=2,3}$. Then
\[
\det(Y)=\det(\widetilde{Y}).
\]
Note that $\widetilde{Y}$ can be viewed as a transversal $(1,1)$-tensor along $\gamma\subset M$. In the local coordinates $(y^1,y^2,y^3)$ on $(\Omega,g)$, $\widetilde{Y}$ can be written as
\[
\widetilde{Y}_{1}^{\phantom{i}1}=0,\quad\widetilde{Y}_{1}^{\phantom{i}j}=0,\quad\widetilde{Y}_{i}^{\phantom{i}1}=0,\quad\widetilde{Y}_{i}^{\phantom{i}j}=Y_{i}^{\phantom{i}j},\quad i,j=2,3.
\]
 See \cite{dahl2009geometrization,feizmohammadi2019inverse} for more details.
Thus $\widetilde{Y}$ solves the equation, written in its invariant form,
\begin{equation}\label{eq_Ytilde}
\frac{D^2}{\mathrm{d}t^2}\widetilde{Y}_{i}^{\phantom{i}j}+R_{ik\ell}^{\phantom{1\alpha1}m}\dot{\gamma}^k\dot{\gamma}^\ell\widetilde{Y}_{m}^{\phantom{\alpha}j}=0.
\end{equation}
To summarize, we now have the \textit{Jacobi weighted ray transform of the first kind} (cf. \cite{feizmohammadi2019inverse}) of $f=\beta c^{3/2}$ along the geodesic $\gamma$ in $(\Omega,g)$ passing through $x_0=\gamma(t_0)$,
\begin{equation}\label{jacobi_transform}
\mathscr{ J}_{\widetilde{Y}}^{(1)}f=\int_{t_-}^{t_+}f(\gamma(t))(\det \widetilde{Y}(t))^{-1/2}\mathrm{d}t,
\end{equation}
for any transversal $(1,1)$-tensor $\widetilde{Y}$ solving \eqref{eq_Ytilde}. If $(\Omega,g)$ has no conjugate points, by \cite[Proposition 3]{feizmohammadi2019inverse}, $f\vert_{\gamma}$ can be recovered by knowing
$\mathscr{J}_{\widetilde{Y}}^{(1)}f$ along the single geodesic $\gamma$ for all admissible $\widetilde{Y}$. In particular, we can determine $\beta(x_0)$. Since $x_0$ is an arbitrary point, $\beta$ can be uniquely determined in $\Omega$.

If $(\Omega,g)$ satisfies the foliation condition described in \cite{paternain2019geodesic}, one can use the invertibility of weighted geodesic ray transform with a single weight established in \cite{paternain2019geodesic}. 
\subsection{The case when $c\equiv 1$}
When the coefficient $c\equiv c_0$ is a constant, without loss of generality, we can assume $c_0=1$. In this case the Riemann curvature tensor $R$ is zero, thus we can take solution to \eqref{eq_Ytilde} to be $\widetilde{Y}_{\alpha}^{\phantom{\alpha}\beta}(t)\equiv\delta_{\alpha}^{\phantom{\alpha}\beta}$, $\alpha,\beta=2,3$, and consequently $\det(\widetilde{Y}(t))\equiv 1$. Now the Jacobi weighted ray transform \eqref{jacobi_transform} reduces to the X-ray transform
\[
Xf(\gamma)=\int_{t_-}^{t_+}f(\gamma(t))\mathrm{d}t,
\]
the invertibility of which is well known.

\section{Reconstruction Algorithm} \label{Section.Numerical}
In this section we propose an algorithm to reconstruct the coefficient of nonlinearity $\beta$ from measured boundary data, and report the results of a numerical implementation of this algorithm.  From now on we assume that $c$ is constant.
Our approach is inspired by our proof of Theorem \ref{mainthm} based on the second order linearization of the problem. However, instead of taking the frequency to infinity later in the proof, we will consider the problem at fixed \textit{high} frequency and take into account the effect of diffraction in this section. 

We assume that $\beta$ is compactly supported in $\Omega$, and fix the value of the frequency $\omega$. %
We follow the notation from Section \ref{Section.SecondLinearization} and take $u_{1} = u_{2} = e^{\mathrm{i} k \cdot x} e^{-\mathrm{i} \omega t}$ representing waves propagating in the $\hat{k}$-direction, with wave number $k$, wave speed $c$, and angular frequency $\omega=c|k|$. The linearized wave field $\mathcal{U}:=\mathcal{U}^{(12)}$ satisfies the equation
\[
(-c^{-2}\partial^2_t+\Delta)\mathcal{U}=8\omega^2\beta e^{2\mathrm{i} k\cdot x} e^{- 2\mathrm{i} \omega t}.
\]
Assuming now that we have measurements for all $t$, and taking Fourier transform in $t$ of the above equation, one gets
\begin{equation}\label{eq_mcU}
\Delta \hat{\mathcal{U}}(x,\xi) + (\xi /c)^2  \hat{\mathcal{U}}(x,\xi) =  8\omega^2\beta e^{2 \mathrm{i} k \cdot x} \delta(\xi - 2 \omega),
\end{equation}
where $ \hat{\mathcal{U}}(x,\xi)=\mathcal{F}_{t\rightarrow\xi}\mathcal{U}$ (in the distributional sense). Notice that, as explained in the introduction, the isonating profiles oscillate at the fundamental harmonic $\omega$ whereas the wave field generated by the nonlinearity oscillates at the second harmonic $2 \omega$ realized in \eqref{helmholtz_eq} by the presence of the delta function $\delta(\xi - 2 \omega)$. We can let $\hat{\mathcal{U}}(x,\xi)=0$ for $\xi\neq 2\omega$. Let $\mathcal{V}(x)=\frac{1}{8\omega^2}\int\hat{\mathcal{U}}(x,\xi)\mathrm{d}\xi$, then one can integrate the equation \eqref{eq_mcU}, and find that  $\mathcal{V}$ satisfies the Helmholtz equation (again in distributional sense)
\begin{equation}\label{helmholtz_eq}
\Delta \mathcal{V} + (2\omega /c)^2  \mathcal{V}=\beta e^{2 \mathrm{i} k \cdot x}.
\end{equation}
Multiply above equation by $\phi(x)=e^{-2\mathrm{i}\theta\cdot x}$ with $|\theta|=|k|$, and integrate by parts, we obtain
\[
\int_{\Omega}\beta e^{2\mathrm{i}(k-\theta)\cdot x}\mathrm{d}x=\int_{\partial \Omega}(\frac{\partial\mathcal{V}}{\partial\nu}\phi-\mathcal{V}\frac{\partial \phi}{\partial\nu})\mathrm{d}S.
\]
Since $\{k-\theta\vert k,\theta\in S^2,|k|=|\theta|\}$ fills in the ball with radius $2|k|$, we can get $\hat{\beta}$ in a ball of radius $4|k|$ from $\mathcal{V}\vert_{\partial\Omega},\frac{\partial\mathcal{V}}{\partial\nu}\vert_{\partial\Omega}$. This shows the uniqueness of $\beta$ from fixed-frequency data, using the fact that $\beta$ is compactly supported.

\subsection{One-way wave approximation in the high frequency regime}
We take $\Omega=(-\frac{L}{2},\frac{L}{2})\times(-\frac{L}{2},\frac{L}{2})$.
In the ultrasound transmission setting, the domain $\Omega$ is isonated on one side (for example, $x=-\frac{L}{2}$) and the transmitted wave generated by the nonlinearity is measured on the opposite side ($x=\frac{L}{2}$). 
To isolate the forward moving part of the wave field, we will use the so-called paraxial, parabolic, or one-way wave approximations intended to describe waves with a preferred direction of propagation \cite{Bamberger1988a,Bamberger1988b,Angus2014,
Antoine1999,Acosta2017c}.

We follow the approach described in \cite{Antoine1999} based on the work of Nirenberg \cite{Nirenberg1973} showing that, up to a smoothing operator, the Helmholtz operator can be factored into its forward and backward components that recast \eqref{helmholtz_eq} as
\begin{equation}\label{helmholtz_eq_factored}
\left( \partial_{x} - \Lambda^{+}_{2\omega} \right) \left( \partial_{x} - \Lambda^{-}_{2\omega} \right)  \mathcal{V} = \beta e^{2 \mathrm{i} k x} 
\end{equation}
where $\Lambda^{\pm}_{2\omega}$ are pseudo-differential operators of order $+1$ known as the forward and backward Dirichlet--to--Neumann maps. These DtN operators $\Lambda^{\pm}_{2\omega}$ are not to be confused with the DtN map $\Lambda$ defined in Section \ref{Section.Intro}. The latter maps Dirichlet data to Neumann data for solutions to the full wave equation \eqref{eq_main} in $\Omega$. The former map Dirichlet data to Neumann data for the respective forward and backward components of the wave field across any vertical line within the domain $\Omega$.

By setting $V = ( \partial_{x} - \Lambda^{-}_{2\omega} )  \mathcal{V}$, we obtain the forward moving wave field satisfying
\begin{equation}\label{forward_eq}
\partial_{x} V - \Lambda^{+}_{2 \omega} V = \beta e^{2 \mathrm{i} k x}.
\end{equation}
Local approximations of the forward DtN map are the subject of much research in the area of absorbing boundary conditions. See \cite{Antoine1999} and references therein. In particular, in the high frequency regime, the forward DtN map has the following approximation,
\begin{equation}\label{DtN_approx}
\Lambda^{+}_{2\omega} V = \frac{2\mathrm{i}\omega}{c} V - \frac{c}{4 \mathrm{i}\omega} \partial_{y}^2 V + \mathcal{O}(1/\omega^2).
\end{equation}
Neglecting the higher order terms and plugging this approximation into \eqref{forward_eq} we obtain our one-way wave model of the ultrasound transmission field $V$ satisfying
\begin{equation}\label{forward_eq_2}
\partial_{x} V - 2\mathrm{i}k V + \frac{1}{4 \mathrm{i} k} \partial_{y}^2 V = \beta e^{2 \mathrm{i} k x}.
\end{equation}
Using the integrating factor $e^{- 2 \mathrm{i} k x}$ and defining $v = V e^{- 2 \mathrm{i} k x}$ we obtain the following governing equation
\begin{equation}\label{forward_eq_3}
\partial_{x} v + \frac{1}{4 \mathrm{i} k} \partial_{y}^2 v =  \beta,
\end{equation}
which is a parabolic equation for propagation in the $x$-direction with complex diffusion in the $y$-direction. This equation is supplemented by vanishing Dirichlet conditions at $x=-L/2$ and $y=\pm L/2$. Also notice that the highly oscillatory factor $e^{- 2 \mathrm{i} k x}$ does not appear in \eqref{forward_eq_3} which is extremely important for computational purposes in order to accurately approximate $v$ over a long distance $L$ relative to the wavelength $\lambda = 2 \pi / k$.

\subsection{Inversion algorithm}
In order to obtain tomographic-like measurements, we replace $\beta$ in \eqref{forward_eq_3} with a rotated version $\beta_{\theta} = \beta \circ R_{\theta}$ where $R_{\theta}$ is the orthogonal matrix for performing a rotation by an angle $\theta$. The solution $v_{\theta}$ to \eqref{forward_eq_3} with $\beta_{\theta}$ as the source renders a mapping $\beta \mapsto v_{\theta}|_{x=L/2}$. This mapping is analogue to the Radon transform of $\beta$ but for probing the domain of interest with waves at finite frequencies,
\begin{equation}\label{Forward1}
\mathcal{W}[\beta](\theta,y) = v_{\theta}(L/2, y).
\end{equation}
We will propose an algorithm for finding $\beta$ given the measured data encoded in $\mathcal{W}[\beta]$.

Our proposed algorithm to invert \eqref{Forward1} is analogue to the filtered back-projection employed to invert the Radon transform \cite{Natterer2001book,Feeman2015book}. We propose 
\begin{equation}\label{Propose_algo_1}
\beta \approx \mathcal{W}^{*} \left[ h \ast_{y} \mathcal{W} [\beta] \right]
\end{equation}
where $h$ is a suitable filter acting on the $y$-variable, $\ast_{y}$ represents convolution in the $y$-variable, and $\mathcal{W}^{*}$ is the adjoint of $\mathcal{W}$ with respect to the $L^2$ inner product. For \eqref{Propose_algo_1} to be applicable, it only remains to find an implementable expression for $\mathcal{W}^{*}$. This can be accomplished by multiplying \eqref{forward_eq_3} by $\overline{\varphi}_{\theta}$ and integrating by parts to yield
\begin{equation}\label{Adjoint_W}
\int \int \mathcal{W} [\beta]  \overline{\eta} \, dy \, d \theta = \int_{\Omega} \beta  \int  \left( \overline{\varphi}_{\theta} \circ R_{\theta}^{T} \right) \, d \theta \, d(x,y)
\end{equation}
where $\varphi_{\theta}$ satisfies
\begin{equation}\label{Adjoint_phi}
\partial_{x} \varphi_{\theta} + \frac{1}{4\mathrm{i}k} \partial_{y}^2 \varphi_{\theta} = 0
\end{equation}
backwards in the $x$-direction, augmented by vanishing Dirichlet condition at $y=\pm L/2$ and prescribed Dirichlet condition $\varphi_{\theta} = \eta = \eta(\theta,y)$ at $x = L/2$. The identity \eqref{Adjoint_W} shows that
\begin{equation}\label{Adjoint_phi2}
\mathcal{W}^{*}[\eta] = \int \left( \varphi_{\theta} \circ R^{T}_{\theta} \right) \, d \theta.
\end{equation}

\subsection{Numerical examples}
We implemented the formula \eqref{Propose_algo_1} on a discrete setting to reconstruct the Shepp-Logan and an image of the brain vasculature. In each case, we rotated the image $\beta$ with an angular resolution of one degree for $\theta \in (0^{\circ} , 360^{\circ})$. The true images are displayed in Figure \ref{fig.TrueImages} (top row). For the numerical implementation, we solved \eqref{forward_eq_3} using a centered second order finite difference scheme to approximate $\partial_{y}^2$ and the Crank–Nicolson scheme for $\partial_{x}$. This method is well-known to be unconditionally stable \cite[\S 6.3]{Strikwerda-Book-2004}. For $\theta=0$, the amplitude of the numerical solution $v$ to \eqref{forward_eq_3} is shown in Figure \ref{fig.TrueImages} (bottom row) for $L/\lambda = 100$ where $\lambda$ is the wavelength of the probing field. We observe the spreading of the wave field inherent in the propagation of waves at a finite frequency. 

We numerically simulated the probing of the domain of interest with waves at frequencies such that $L/\lambda = 10$ and $L/\lambda = 100$, and using the Ram-Lak filter $h$ implemented using the Fourier transform with respect to $y$. The synthetic measurements and reconstructed images are shown in Figure \ref{fig.ReconSheppLogan} for the Shepp-Logan phantom and in Figure \ref{fig.ReconBrain} for the brain vasculature. Note that even for $L/\lambda = 100$, there is considerable wave spreading contained in the measurements. However, the proposed algorithm is able to back-propagate the wave spreading in order to form a relatively sharp image.

\begin{figure}[ht!]
\includegraphics[width=0.48 \textwidth, trim = -5 -5 -5 -5, clip]{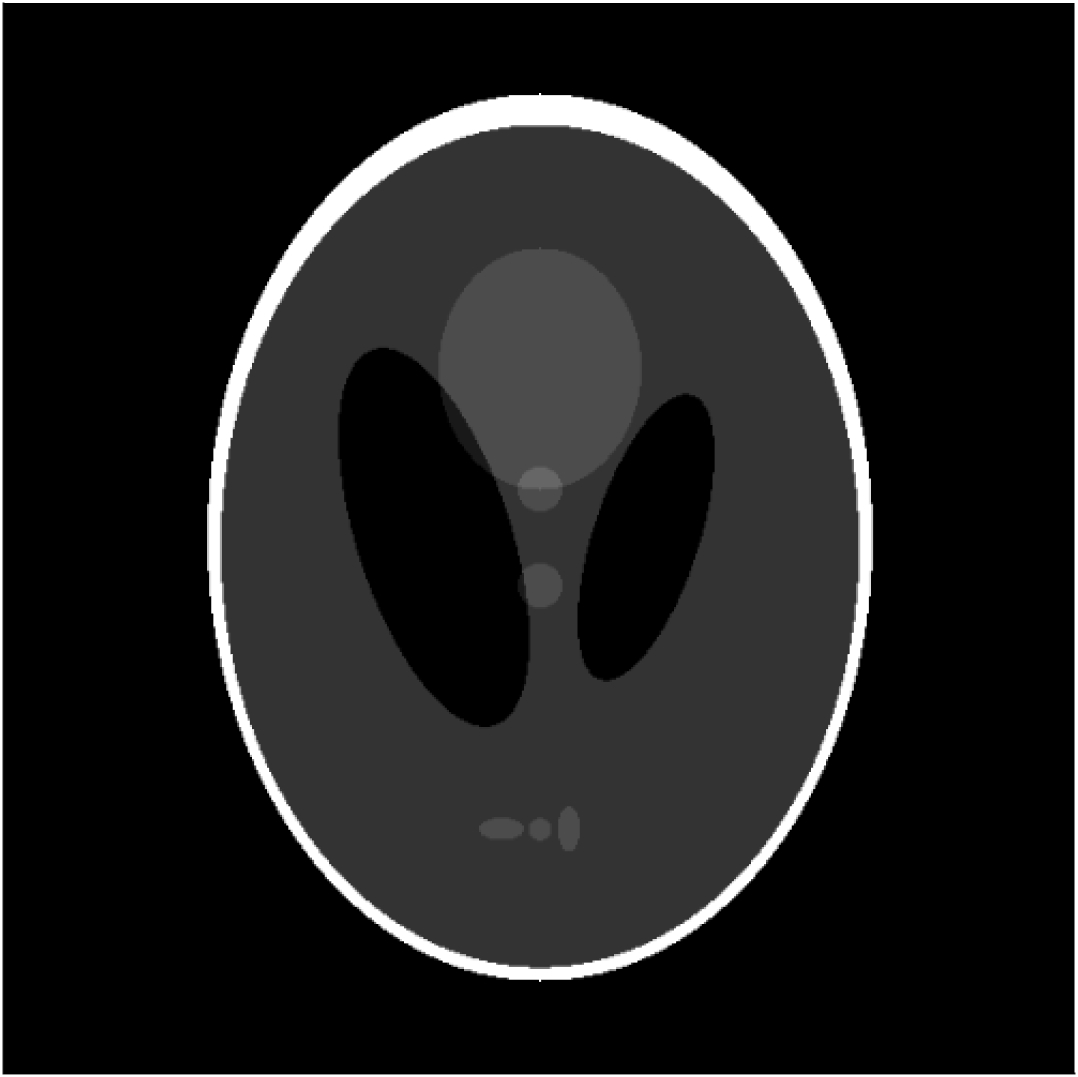} 
\includegraphics[width=0.48 \textwidth, trim = -5 -5 -5 -5, clip]{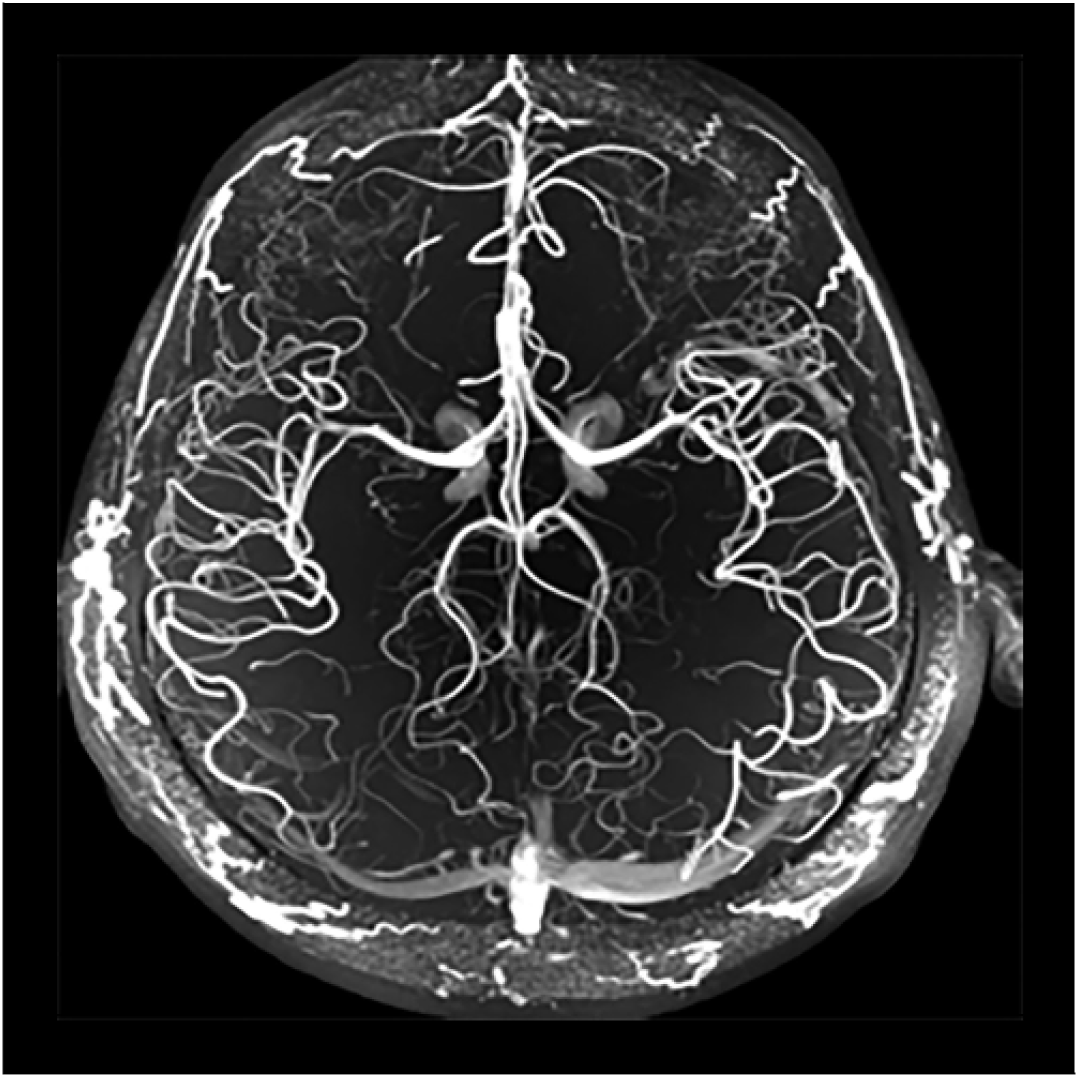} \\
\includegraphics[width=0.48 \textwidth, trim = -5 -5 -5 -5, clip]{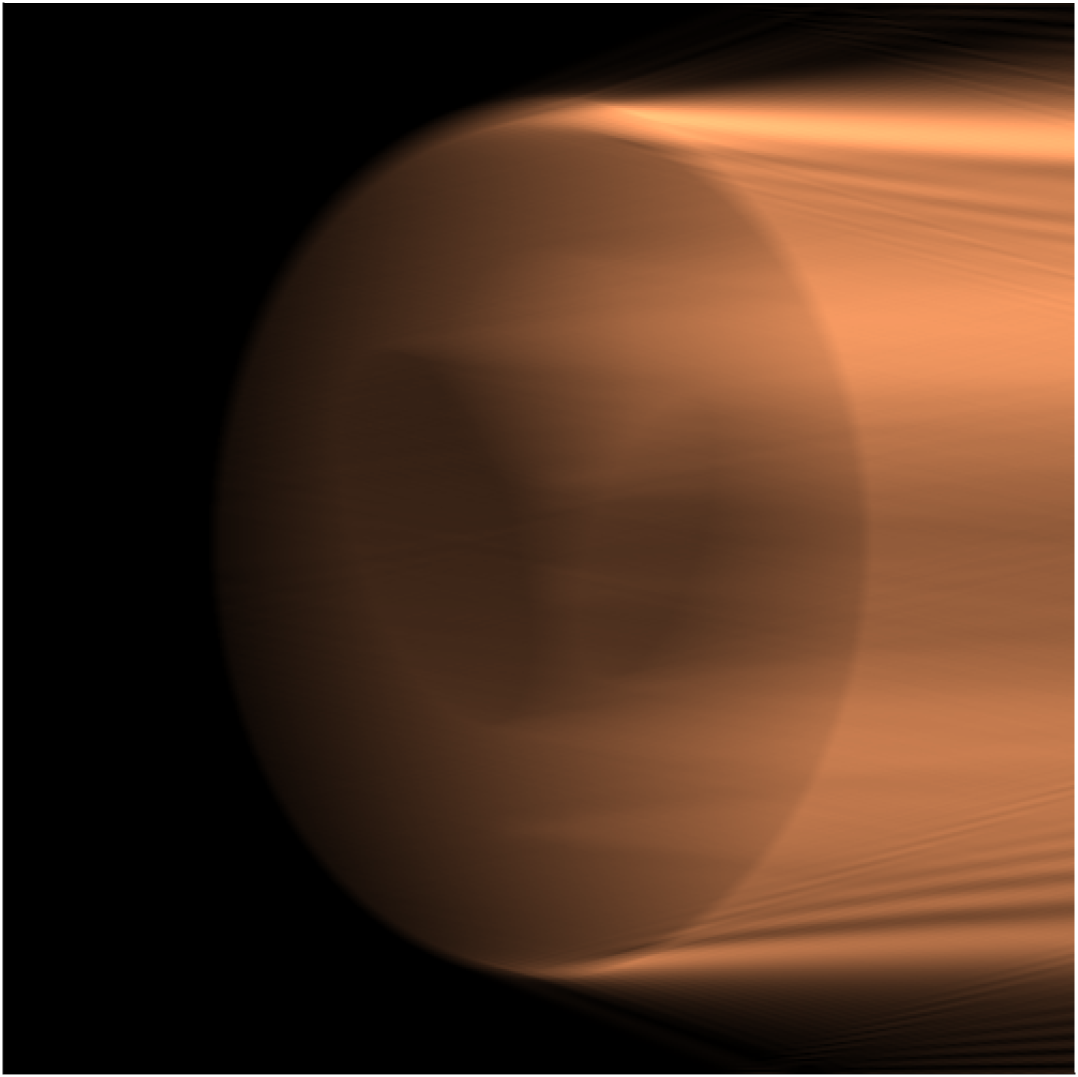} 
\includegraphics[width=0.48 \textwidth, trim = -5 -5 -5 -5, clip]{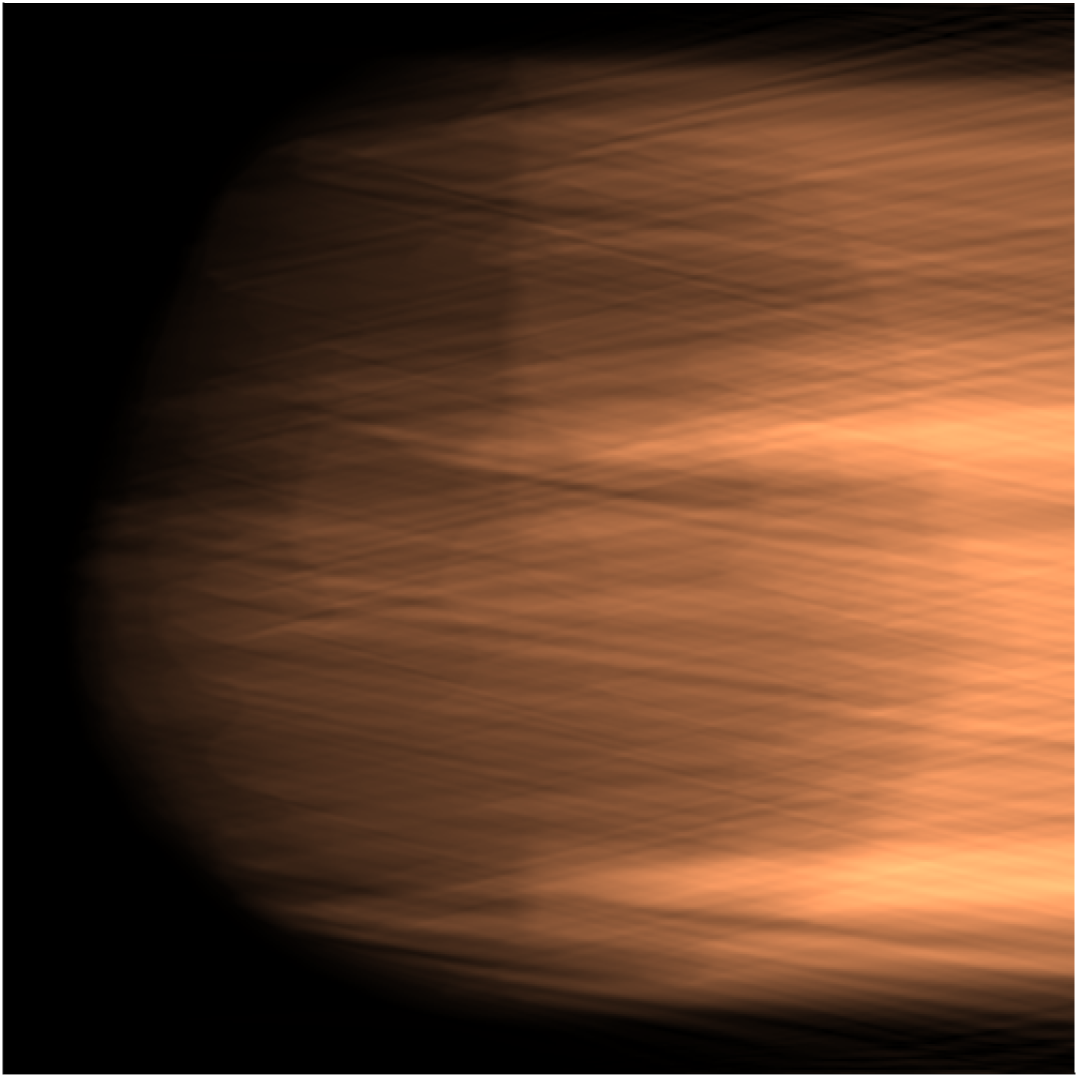}
\caption{True profiles for the Shepp-Logan phantom (top left) and brain vasculature (top right) obtained from \textnormal{\cite{Koroshetz2018}}. The amplitude of the numerical solutions $v$ for $\theta=0$ to \eqref{forward_eq_3} (bottom row) for $L/\lambda = 100$ where $\lambda$ is the wavelength of the probing field.}
\label{fig.TrueImages}
\end{figure}

\begin{figure}[ht!]
\centering
\includegraphics[width=0.48 \textwidth, trim = -5 -5 -5 -5, clip]{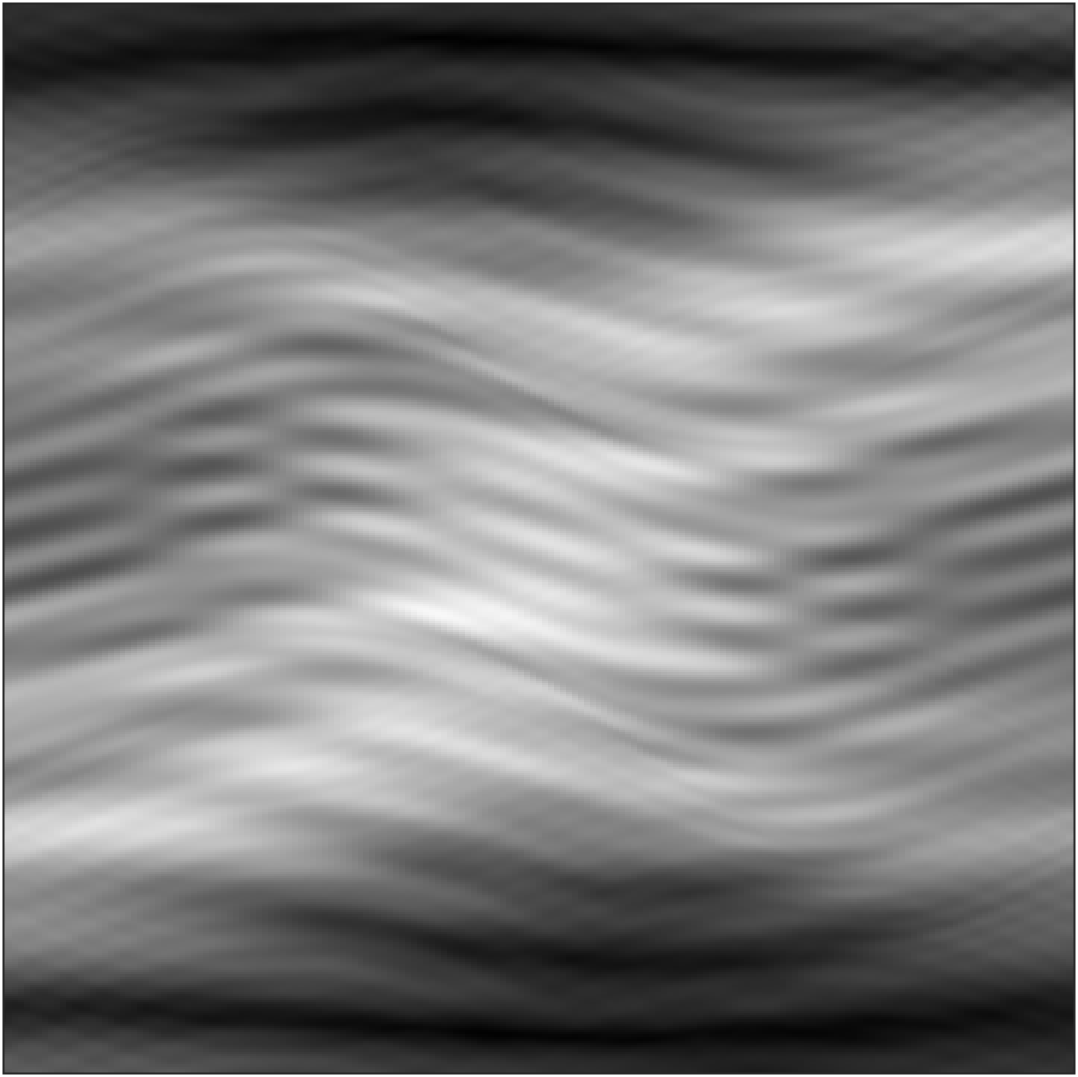} 
\includegraphics[width=0.48 \textwidth, trim = -5 -5 -5 -5, clip]{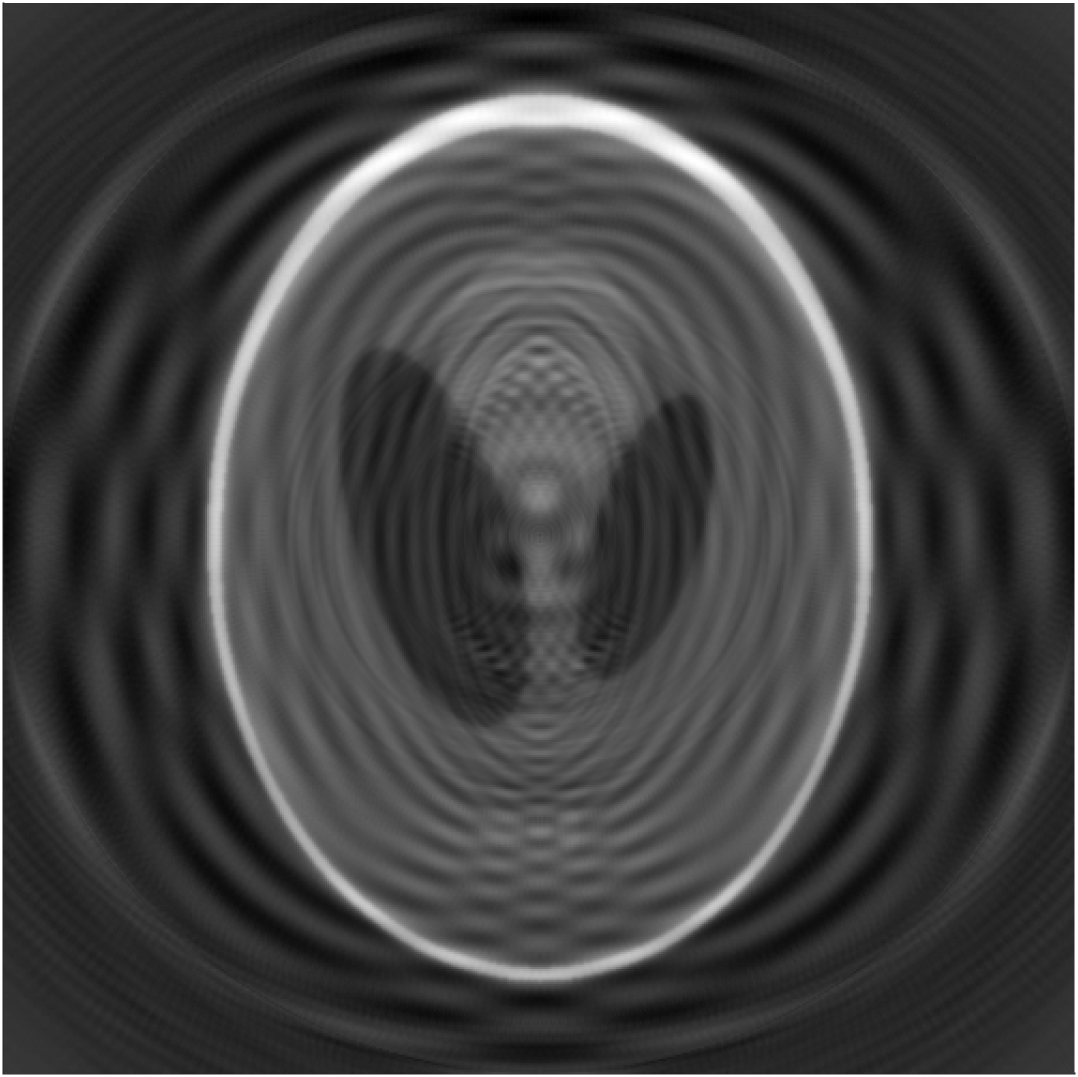} \\
\includegraphics[width=0.48 \textwidth, trim = -5 -5 -5 -5, clip]{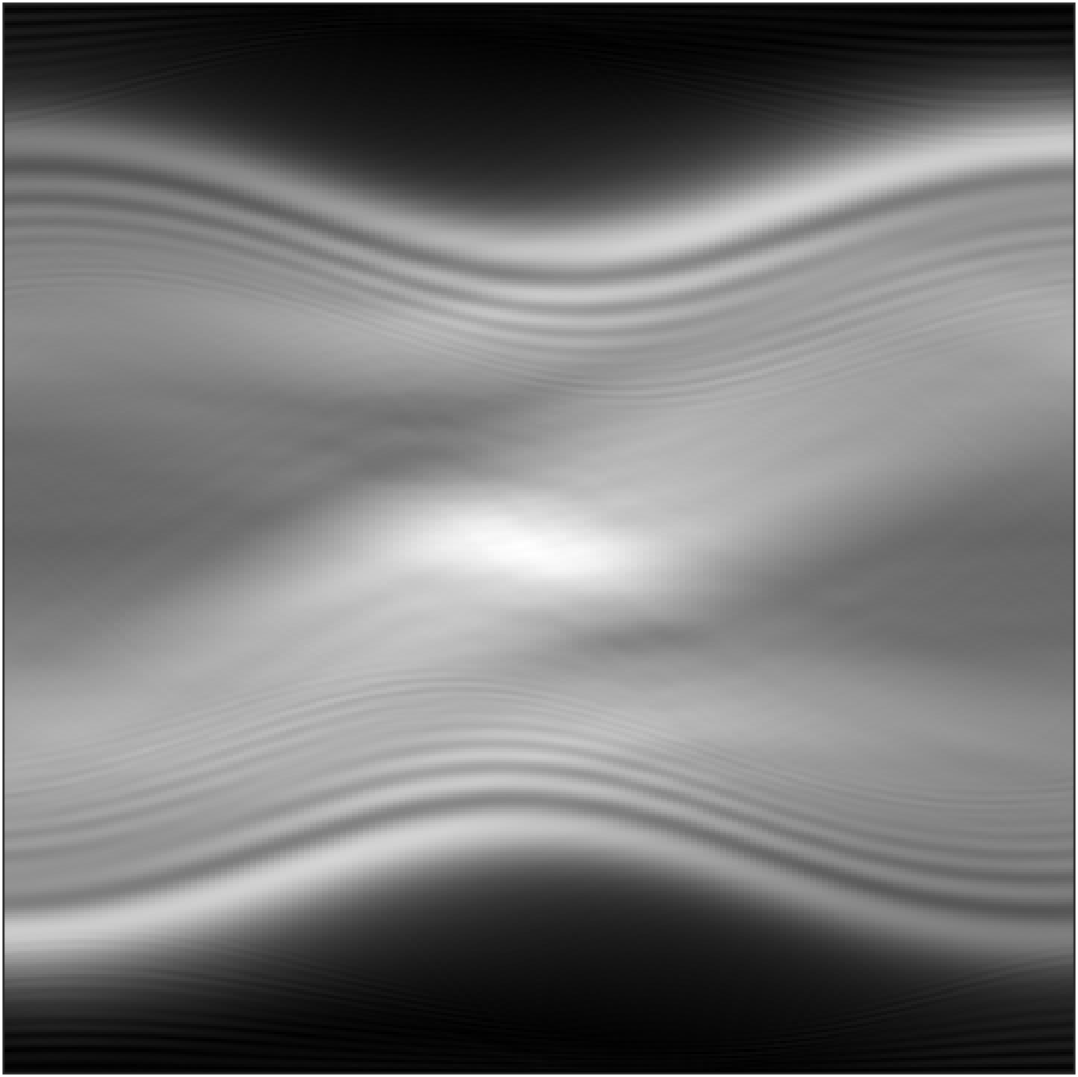} 
\includegraphics[width=0.48 \textwidth, trim = -5 -5 -5 -5, clip]{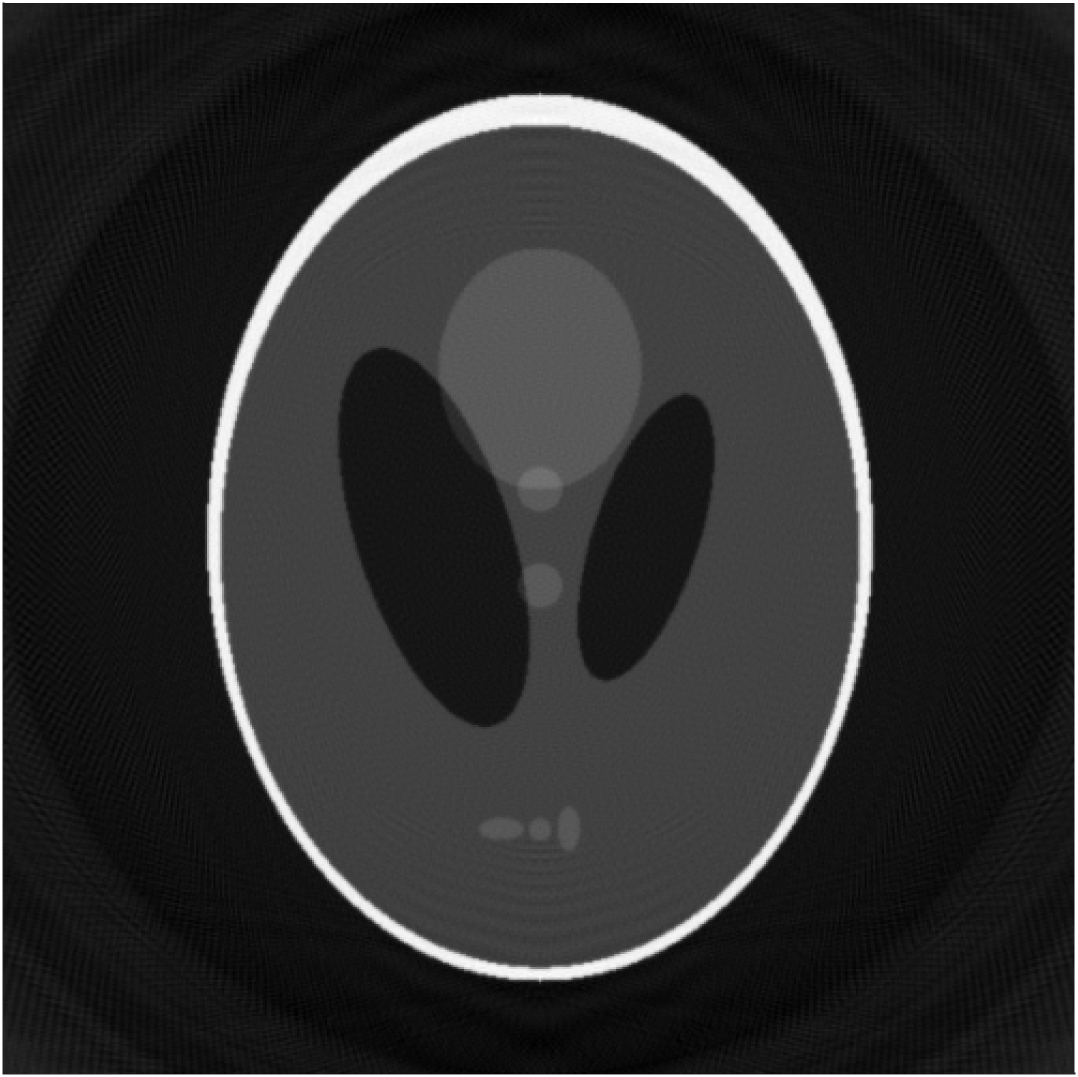}
\caption{Synthetic data using the forward model \eqref{Forward1} and reconstructed profiles using the inversion algorithm \eqref{Propose_algo_1} for 
probing ultrasound waves of frequencies such that $L/\lambda = 10$ (top row) and $L/\lambda = 100$ (bottom row) where $L$ is the size of the image and $\lambda$ is the wavelength.}
\label{fig.ReconSheppLogan}
\end{figure}

\begin{figure}[ht!]
\includegraphics[width=0.48 \textwidth, trim = -5 -5 -5 -5, clip]{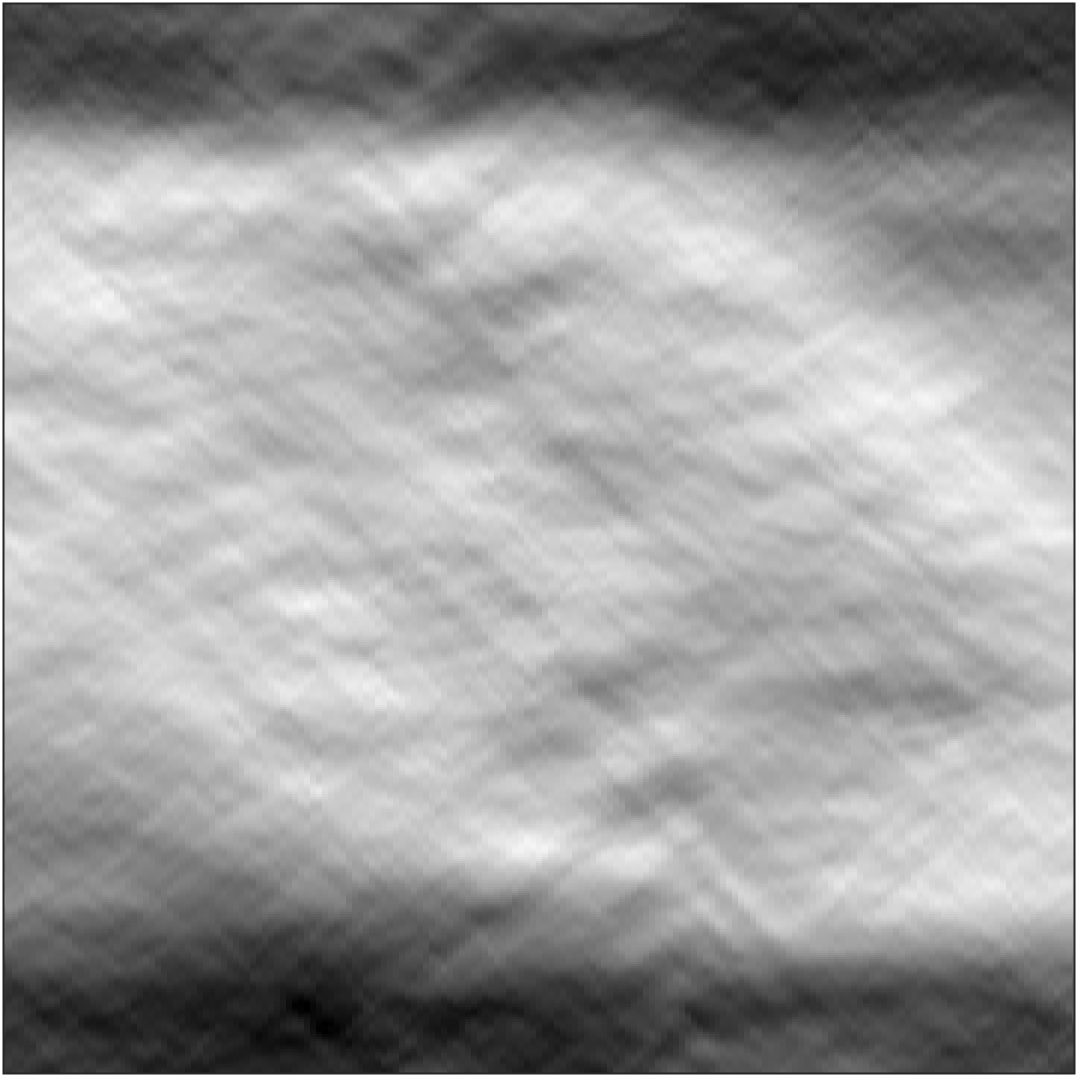} 
\includegraphics[width=0.48 \textwidth, trim = -5 -5 -5 -5, clip]{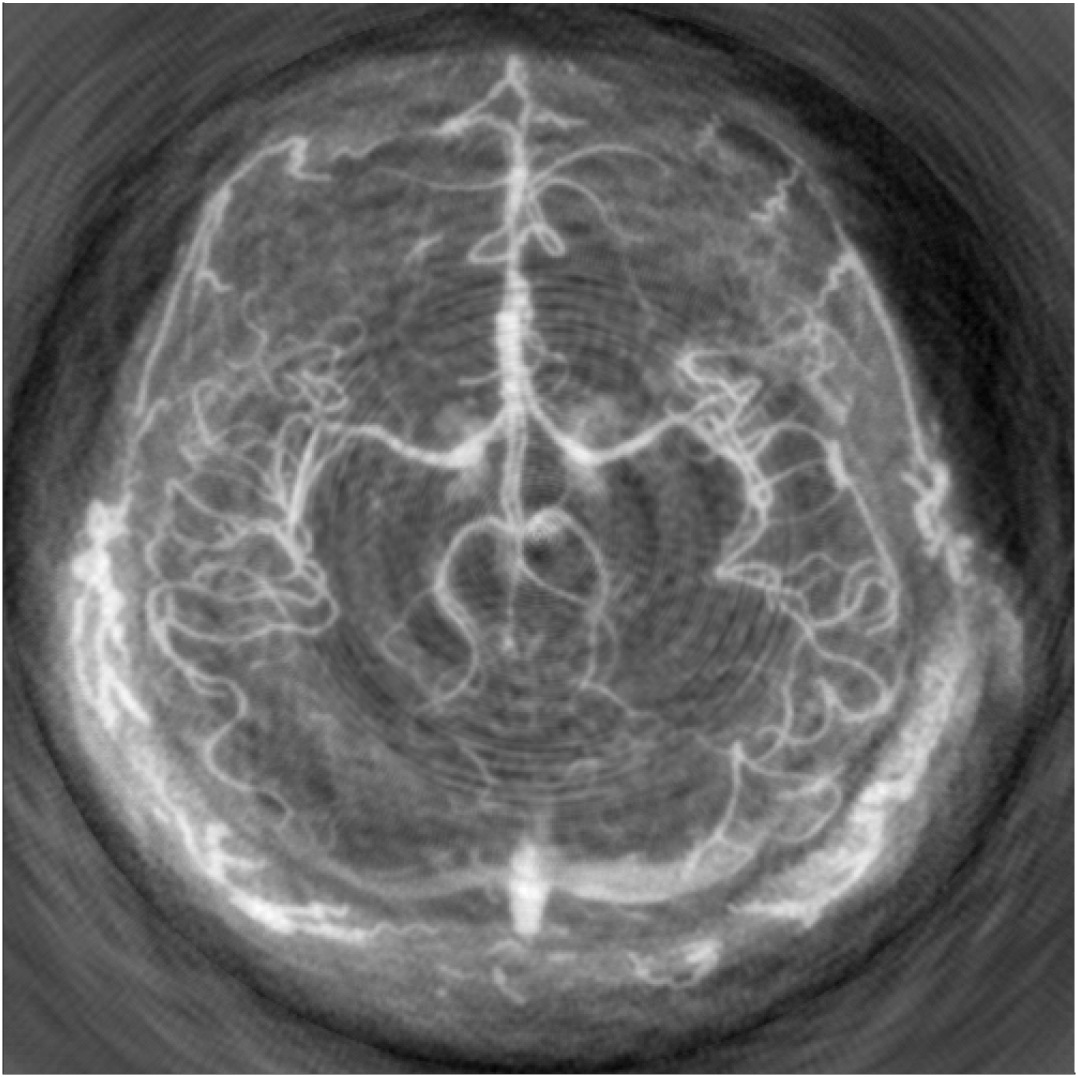} \\
\includegraphics[width=0.48 \textwidth, trim = -5 -5 -5 -5, clip]{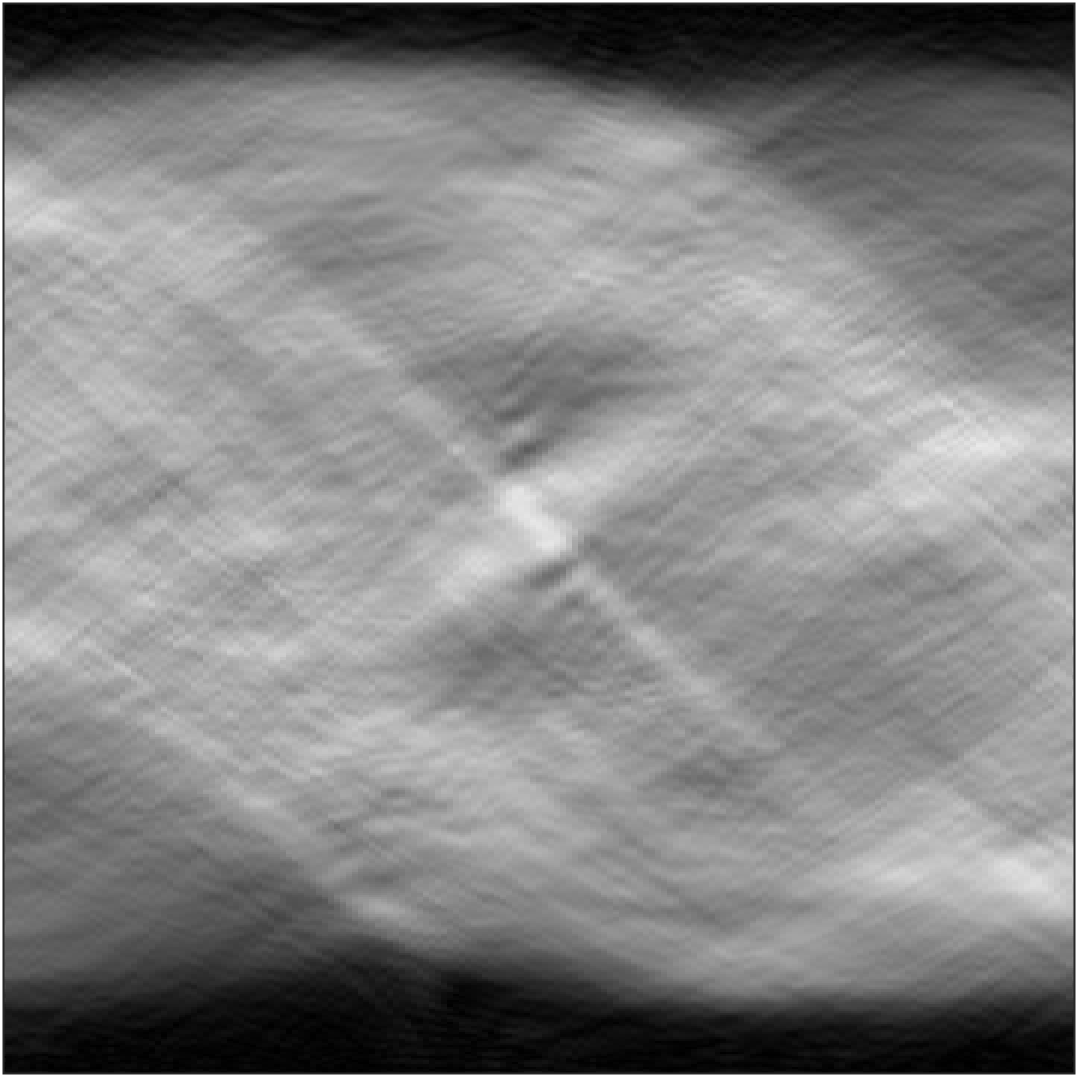} 
\includegraphics[width=0.48 \textwidth, trim = -5 -5 -5 -5, clip]{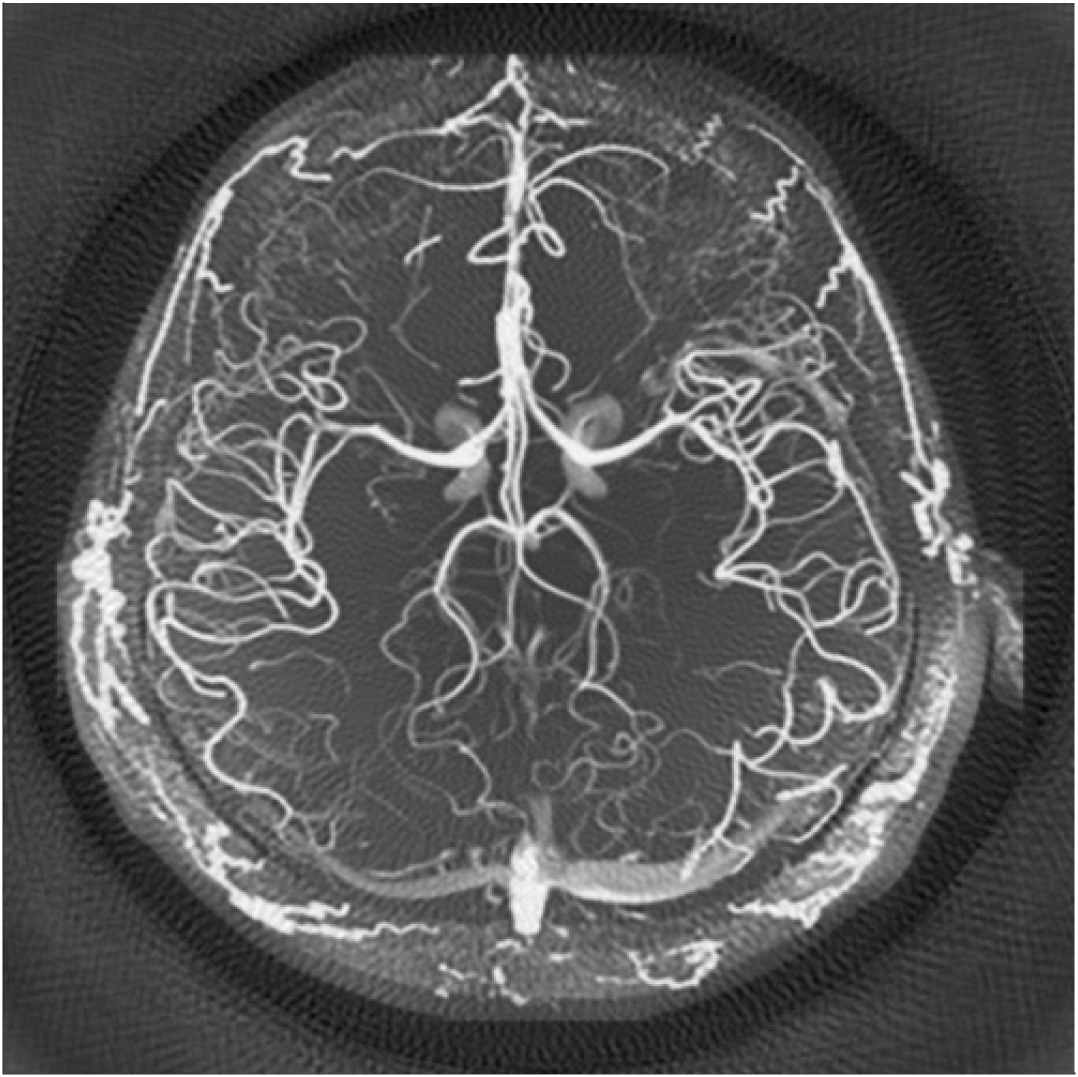} 
\caption{Synthetic data using the forward model \eqref{Forward1} and reconstructed profiles using the inversion algorithm \eqref{Propose_algo_1} for 
probing ultrasound waves of frequencies such that $L/\lambda = 10$ (top row) and $L/\lambda = 100$ (bottom row) where $L$ is the size of the image and $\lambda$ is the wavelength.}
\label{fig.ReconBrain}
\end{figure}

\section{Conclusion} \label{Section.Conclusion}

We have shown that the coefficient of nonlinearity in a Westervelt wave equation can be uniquely recovered from knowledge of the Dirichlet-to-Neumann map. This result supports the reliability of forming quantitative images of blood perfusion using microbubble enhancing agents to trigger the formation of nonlinear ultrasound waves. The nonlinear term in the Westervelt equation has been shown to accurately model the resonant properties of microbubble enhancing agents under the passage of ultrasound waves. See \cite{Mischi2014} and the references therein. We proposed an inversion formula expected to be accurate in the high frequency regime. This formula mimics the inversion of the Radon transform, but takes into account the spreading of the wave beams observed at finite frequencies. As a proof-of-concept, we numerically implemented the reconstruction algorithm and showed its ability to back-propagate wave spreading and form sharp reconstructed images.

\section*{Acknowledgement}
SA would like to thank the Texas Children's Hospital for its support and research-oriented environment. 
GU was partly supported by
NSF,  the Walker Family Endowed Professorship at UW and the Si-Yuan
Professorship at IAS, HKUST. JZ was partially supported by Research Grant Council of Hong Kong (GRF grant 16305018). 
\bibliographystyle{siamplain}
\bibliography{biblio,NonlinearUS}
\end{document}